\documentclass[letterpaper, 11pt,  reqno]{amsart}

\usepackage{amsmath,amssymb,amscd,amsthm,amsxtra, esint, bbm, enumerate, relsize, xcolor}
\usepackage{mathtools} 
\usepackage{bbm} 
\usepackage{accents, cancel}
\usepackage[implicit=true]{hyperref}

\usepackage{stmaryrd}

\usepackage{mparhack}

\allowdisplaybreaks[2]

\newtheorem{theorem}{Theorem} [section]

\newtheorem{lemma}[theorem]{Lemma}
\newtheorem{proposition}[theorem]{Proposition}

\newtheorem{corollary}[theorem]{Corollary}

\theoremstyle{definition}
\newtheorem{remark}[theorem]{Remark}
\newtheorem{assumption}{Assumption}
\newtheorem{definition}[theorem]{Definition}

\newcommand{\1}{\mathbbm 1}

\newcommand{\Z}{\mathbb{Z}}
\newcommand{\R}{\mathbb{R}}

\newcommand{\T}{\mathbb{T}}

\let \div \relax

\let\Re=\undefined\DeclareMathOperator*{\Re}{Re}
\let\Im=\undefined\DeclareMathOperator*{\Im}{Im}

\newcommand{\E}{\mathbb{E}}

\newcommand{\FL}{\mathcal{F}L} 

\newcommand{\ep}{\varepsilon}

\newcommand{\s}{\sigma}

\newcommand{\ft}{\widehat}

\newcommand{\cj}{\overline}

\newcommand{\les}{\lesssim}

\newcommand{\jb}[1]
{\langle #1 \rangle}

\newcommand{\ind}{\mathbbm 1}

\newcommand{\N}{\mathbb{N}}

\newcommand{\QQ}{\mathcal{Q}}

\newcommand{\eps}{\ep}

\newcommand{\ph}{\varphi}

\renewcommand{\ft}{\widehat}

\numberwithin{equation}{section}
\numberwithin{theorem}{section}

\let\div\relax

\DeclareMathOperator{\div}{div}

\makeatletter
\@namedef{subjclassname@2020}{%
  \textup{2020} Mathematics Subject Classification}
\makeatother

\begin{document}
\baselineskip = 14pt

\title[Quasi-invariance for NLS on $\mathbb T^2$]
{Quasi invariant Gaussian measures for the nonlinear Schr\"odinger equation on $\mathbb T^2$} 

\begin{abstract}
    We study the transport of Gaussian measures under the flow of the 2-dimensional defocusing Schr\"odinger equation $i \partial_t u + \Delta u = |u|^{2k} u$ posed on $\mathbb T^2$. In particular, we show that the Gaussian measures with inverse covariance $\|u\|_{H^s}^2$, are quasi-invariant under the flow for $s>2$. Moreover, we show that the Radon-Nykodim density belongs to every $L^p$ space, locally in space. The proof relies on the physical-space energies introduced in [52], as well as a new abstract quasi-invariance argument that allows us to combine space-time estimates, along the flow with probabilistic bounds on the support of the measure. 
\end{abstract}

\author[L.~Tolomeo, N.~Visciglia]
{Leonardo Tolomeo and Nicola Visciglia}

 \address{
Leonardo Tolomeo\\
School of Mathematics\\
The University of Edinburgh\\
and The Maxwell Institute for the Mathematical Sciences\\
James Clerk Maxwell Building\\
The King's Buildings\\
Peter Guthrie Tait Road\\
Edinburgh\\ 
EH9 3FD\\
 United Kingdom\\
}

\email{l.tolomeo@ed.ac.uk}

\address{
Nicola Visciglia\\
Dipartimento di Matematica\\
Universit\`a di Pisa\\
Largo B. Pontecorvo 5, 56127\\
Italy\\
}

\email{nicola.visciglia@unipi.it}


\maketitle
\section{Introduction}
In this paper we study the Cauchy problem associated with the defocusing Nonlinear Schr\"odinger Equation posed on $\T^2$, with a generic (defocusing) odd nonlinearity:
\begin{equation}
    \begin{cases} \label{NLS}
i\partial_t u + \Delta u=|u|^{2k}u, \quad k\in \N, \quad (t, x)\in \R \times \T^2\\
u(0,x)=u_0\in H^\s,
\end{cases}
\end{equation}
where $H^\s$ denotes the usual Sobolev space on $\T^2$. As a byproduct of the paper \cite{BGT}, one can show that the Cauchy problem above is globally well posed in the Sobolev spaces $H^\s$ with $\s\geq 1$, hence it admits a well-defined flow $\Phi_t:H^\s\rightarrow H^\s$. In this work, we study how certain Gaussian measures are transported under this flow. 
More precisely, we assume that the initial datum is distributed along the Gaussian measure 
$\mu_s$, formally given by $"\frac 1{\mathcal Z} e^{-\frac 12 \|u\|^2_{H^s}} du"$, induced by the random vector 
\begin{equation}\label{defmus}
\omega\mapsto \phi^\omega(x)=\sum_{n\in \Z^2} \frac{g_n(\omega)}{{\jb n}^s} e^{in\cdot x},\end{equation}
where $\{g_n\}_{n\in \Z^2}$ is an i.i.d. family of complex gaussian variables such that $\E[g_n]=0, \E[|g_n|^2]=1, 
\E[g_n^2]=0$.
The main result of this paper is the following.
\begin{theorem} \label{thm:main}
Let $s>2$, and let $\Phi_t$ be the flow of \eqref{NLS} on the Sobolev space
$H^\s$ with $1\leq \s<s-1$. For every $t\in \R$, we have that
$$\Phi_{t\#}\mu_s \ll \mu_s.$$
Moreover, there exists a Banach space $X \supseteq H^1$, such that for every $t \in \R$, $\Phi_{t\#}\mu_s(X) = 1$, and moreover, for every $p > 1$, and for every $R, T >0$ the Radon-Nykodym derivative 
$$ f_t:= \frac{d \Phi_{t\#}\mu_s}{d \mu_s}$$ satisfies 
\begin{equation} \label{mudensitybound}
    \sup_{|t| \le T} \big \| f_t \1_{\|u\|_{X} \le R}\big \|_{L^p(\mu_s)} < \infty.
\end{equation}
\end{theorem}
Quasi-invariance for the deterministic flow of Hamiltonian and dispersive PDEs in situations where the initial data is {not} distributed according to an invariant measure has been studied extensively in the last decade. This pursuit was initiated by Tzvetkov in \cite{Tz} for the BBM equation, with many results following for a number of other Hamiltonian PDEs, see
\cite{Bo96, BTh, CT, DT2020, FS, FT, FT2, GLT1, GLT2, GOTW, Knezevitch1, Knezevitch2, Knezevitch3, OS, OST, OTT, OTz, OTz2, PTV, PTV2, ST, STX, ThTz}
and references therein. Recently, the first author has kick-started the process of studying similar properties in the setting of (Hamiltonian) PDEs with the addition of stochastic noise, see \cite{FT3} and the related work \cite{CHT}.
\subsection{Background and proof strategy}
We now describe the ``standard" strategy to show results such as Theorem \ref{thm:main}, and the main technical novelties of this work. Consider (as a toy model) the dynamical system on $\R^d$
\begin{equation} \label{ODE}
\begin{cases}
    \dot y = b(y), \\
    y(0) = x,
\end{cases}
\end{equation}
where $b: \R^d \to \R^d$ is a smooth vector field. For the purpose of this subsection, denote by $\Phi_t(x)$ the solution of \eqref{ODE}. 
For every probability measure $\nu_0$ on $\R^d$, 
the transported measure $\nu_t := (\Phi_t(x))_{\#} \nu_0$ will satisfy the Liouville's equation
\begin{equation}\label{Liouville}
    \partial_t \nu_t = - \div{(b\, \nu_t)}.
\end{equation}
When $\nu_0$ is a probability measure of the form 
$$ \nu_0 = \frac{1}{Z}\exp(-E(y)) dy $$
for some smooth function $ E: \R^d \to \R$ with $\exp(-E) \in L^1(\R^d)$, then we can solve \eqref{Liouville} for 
$$ f_t:= \log \frac{d\nu_t}{dx} - \log \frac{d\nu_0}{dx} $$
explicitly with the method of characteristics, and obtain that 
$$ f_t = E(x) -E(\Phi_{-t}(x)) + \int_0^t \div{(b)}\big(\Phi_{-t'}(x)\big) dt'. $$ 
For convenience, define 
\begin{equation}
    \QQ(x) = \left.\frac{d}{dt} E(y(t))\right|_{t=0} + \div{(b)},
\end{equation}
from which we obtain 
\begin{equation} \label{densityformula}
    \nu_t = \exp\Big(\int_0^t \QQ(\Phi_{-t'}(x)) dx \Big) \nu_0.
\end{equation}
The main strategies to prove quasi-invariance so far have been 
\begin{enumerate}
    \item If one can prove that $\QQ \in \exp(L)(\nu_0),$ in the sense that 
    $$ \int \exp(\eta \QQ) d \nu_0 < \infty$$
    for some $\eta > 0$, then one can perform a Gronwall/Yudovich argument on the solution of \eqref{Liouville}, and obtain the existence of the density at every time. This is essentially the strategy followed in the pioneering work \cite{Tz}.
    \item If on top of this, one has good global control on the flow $\Phi_t$, then the condition above can be relaxed to 
    \begin{equation}\label{in0}
        \QQ\in \exp(L)_{\rm loc}(\nu_0).
    \end{equation} 
    This approach was firstly used in the works by Planchon, Tzvetkov, and the second author \cite{PTV, PTV2}. 
    \item Alternatively, one could try to exploit the cancellations caused by the time integral in \eqref{densityformula}, which boil down to showing that
    \begin{equation}
        \label{withdynamics} \int_0^t \QQ(\Phi_{t'}(x)) dx \in \exp(L)_{\rm loc}(\nu_0). 
    \end{equation}
    This was firstly used by Debussche and Tsutsumi in \cite{DT2020} exploiting deterministic, global-in-time bounds on the flow, and subsequently by the first author and Forlano in \cite{FT2} while exploiting probabilistic, local-in-time bounds.
\end{enumerate}

The main novelty of this paper is that we develop a framework to exploit (2) and (3) at the same time. Indeed, speaking informally, in Section \ref{sec:abstract},  we show that if one can write $\QQ = \QQ_1 + \QQ_2$ such that $\QQ_1$ satisfies \eqref{in0} and $\QQ_2$ satisfies \eqref{withdynamics}, then this is enough to obtain quasi-invariance. 

From a technical point of view, in order to simplify the formula for the time derivative $\QQ = \left.\frac{d}{dt} E\right|_{t=0}$, 
it turns out that instead of considering $ E = \frac{1}{2} \|u\|_{H^s}^2$ (which corresponds exactly to $\nu_0 = \mu_s$), it is convenient to add a correction term, and define 
$$ \nu_0 = \exp(-S(u)) d \mu_s.$$
Since $\nu_0$ and $\mu_s$ are equivalent, the respective quasi-invariance statements are also equivalent. 

The construction of the modified energy $S(u)+\frac12 \| u\|_{H^s}^2$ is intimately related to the study of the deterministic growth of Sobolev norms for solutions belonging to the space $H^s$. 
In particular, via the technique of modified energy, it has been shown that many dispersive equations enjoy polynomial growth of the Sobolev norms, see \cite{BPTV, PTV2, PTV_first, PTVHarm}. Recently, using the same strategy, in \cite{HOV} the authors proved a globalization argument for solutions to the generalised derivative NLS.

The energy that we use in this paper was firstly introduced by Planchon, Tzvetkov and the second author in \cite{PTV_first} (in the case $s$ integer), where it was tailored in such a way to obtain the following estimate
$$\big|\frac d{dt} \big(\frac12 \| u\|_{H^s}^2+ S(u)\big)\big|
\les \|u\|_{H^s}^{2-},$$
which in turn implies polynomial growth of the $H^s$ norm (provided that $S(u)$ is lower order term compared to
$\| u\|_{H^s}^2$). 
The main feature that allows to obtain this result is that, in the formula for $\QQ$, the $2s$ derivatives coming from the $\| u\|_{H^s}^2$ are ``well-distributed", see \eqref{energyestimate}.
As it turns out, 
this is also the main property that allows one to obtain the bound \eqref{withdynamics} for almost all the terms in $\QQ$, 
via \emph{deterministic} transference-type multilinear bounds. Nevertheless, due to the low regularity of the Gaussian data ($u \in H^{s-1-\eps}$), one of the terms in $\QQ$ is proportional to
$$ \big(\int D^s \cj u \nabla u\big)\cdot \big(|u|^{2k-2}u^2 D^{s-2}\cj u \nabla u\big), $$
which cannot have a deterministic bound at the regularity of the Gaussian measure, and forces us to prove an estimate of the form \eqref{in0} for this term. 

\subsection{Remarks}
\begin{remark}
    A crucial feature of the quantity 
    $$ \QQ = \left.\frac{d}{dt} \big(S(u) + \frac12 \|u\|_{H^s}^2\big)\right|_{t=0}$$
    is that in many terms, the factors with the most derivatives both depend on $\cj{u}$. See for instance 
    $$ \eqref{E:2} = -2k (k-1)
\int D^{s-2} \bar u u^{k+1}  \Delta \bar u \bar u^{k-2} D^s \bar u.$$
This allows us to exploit the oscillations coming from the linear propagator $e^{-it\Delta}$ in a very strong way, see Proposition \ref{prop:butchercrispi}.
\end{remark}
\begin{remark}
After this paper was completed we have learned that a similar result, with a different proof, has been obtained in the paper \cite{ST2}.    
\end{remark}
\subsection{Notations} We 
denote by $H^\s$ and $W^{s,p}$ the usual Sobolev spaces on the torus $\T^2$. 
We use the Fourier-Lebesgue spaces
$\mathcal FL^{s,\infty}$,
whose norm is defined as
$$\|f\|_{\mathcal FL^{s,\infty}}=\sup_{j\in \Z^2} \hbox{ }  \langle j\rangle^s |\hat f(j)|,$$ where
$\hat f(j)= \int_{\T^2} e^{-ij\cdot x} dx$, $dx$ is the Lebesgue measure on $\T^2$ and $\jb j=\sqrt{1+|j|^2}$. 
We define by $D^s$ the Fourier multiplier associated with $|n|^s$.
For shortness we shall write $\int f=\int_{\T^2} f dx$ of a time-independent function
and $\int_I\int F=\int\int_{I\times \T^2} F(t,x) dxdt$ of any time dependent function $F(t,x)$ 
defined on $I\times \R^2$, with $I\subset \R$. 
For a generic time-dependent function $F(t,x)$
and for a $(X, \|\cdot \|_{X})$, Banach space of time independent functions,
we write $\|F\|_{L^p_I X}^p=\int_I \|F(t,x)\|_{X}^p dt$, with an obvious generalization to the case $p=\infty$.
Given two functions $f$, $g$ time independent, we write
$(f,g)=\int f\bar g $, for the classical $L^2$ scalar product.
For every $a\in \R$ we denote by $a+$
and $a-$ respectively $a+\varepsilon$ and $a-\varepsilon$, for any $\varepsilon>0$ small enough. Given $N \in \{ 1, 2, 4, \dotsc\}$ dyadic, we denote by $P_N$ the sharp Fourier projector given by 
\begin{equation*}
   \ft{P_N f}(n) =  \ft{f}(n) \cdot 
   \begin{cases}
        \1_{|n| \le 1} & \text{ if }N=1,\\
        \1_{N/2 < |n| \le N} & \text{ if }N>1.
    \end{cases}
\end{equation*}
We will also write 
$$ P_{\le N} = \sum_{M\le N} P_N. $$

\subsection*{Acknowledgments} The authors wish to thank the department of Mathematics in Orsay and Nicolas Burq for supporting this research. The authors were partially supported by the European Research Council (ERC) under the European Union's Horizon 2020 research and innovation programme (Grant agreement
101097172 -- GEOEDP).

\section{An abstract criterion for quasi-invariance} \label{sec:abstract}
In this section, we show an abstract quasi-invariance result under the following assumptions.
\begin{assumption}[Energy-type measure]\label{ass:1}
    $H \subseteq X$ be an abstract Wiener space, and let $\mu$ be the cylindrical Wiener measure on $X$, i.e.\ $\mu$ has mean $0$ and covariance 
    \begin{equation}\label{mu_cov}
        \int |\jb{x,\ph}|^2 d\mu(x) = \|\ph\|_{H^*}^2.
    \end{equation}
    for every $\ph \in X^* \subseteq H^*$. For a measurable functional $S: X \to \R$, we define
    \begin{equation} \label{rhoDefAbstract}
        \rho = \exp( - S(x)) \mu(x).
    \end{equation}
\end{assumption}
\begin{assumption}[Flow map]\label{ass:2}
    Let $\Phi: \R \times X \to X$ be jointly continuous flow map, i.e.\ for every $t,s \ge 0, x \in X$, we have that 
    $$ \Phi(t+s, x) = \Phi(t,\Phi(s,x)). $$
    Moreover, we assume that there exists a sequence of finite-dimensional subspaces $E_n \subset H \subseteq X$ and flow maps $\Phi_n: \R_{\ge 0} \times E_n \to E_n$ such that, calling $P_n$ the orthogonal projection to $E_n$, $P_n$ is bounded on $X$ and for every $t \in \R$,
    \begin{equation}
        \Phi(t,x) = \lim_{n \to \infty} \Phi_n(t,P_n(x)).
    \end{equation}
\end{assumption}
\begin{assumption}[Divergence free] \label{divfree}\label{ass3}
    The maps $\Phi_n$ preserve the Lebesgue measure on $E_n$. Moreover,    for every $x \in E_n$, the map $t \mapsto S(\Phi_n(t,x)) + \frac12 \|\Phi_n(t,x)\|_{H}^2$ is $C^1$ in time, and  \begin{equation}\label{norm_derivative}
        \left.\partial_t \left(S(\Phi_n(t,x)) + \frac12 \|\Phi_n(t,x)\|_{H}^2\right)\right|_{t=0} = \QQ_1^n(x) + \QQ_2^n(x).
    \end{equation}
\end{assumption}
\begin{assumption}[Global bounds] \label{ass_global}
    Denoting by 
    $$ B_R = \{ x \in X: \|x\|_X \le R\},$$
    for every $T > 0$, there exists $R'(T)$ such that for every $0 \le t \le R'(T),$ and for every $n\in \N$,
    $$ \Phi_n(t, B_R) \subseteq B_{R'(T)}.$$
\end{assumption}

\begin{assumption}\label{ass5}
    We assume the following on $S$. First of all, for $\mu-a.e.\ x \in X$, we have that 
    \begin{equation} \label{Rconvergence}
        S(P_n(x)) \to S(x).
    \end{equation}
    Moreover, for every ball $B_R$, and some $p> 1$, we have that 
    \begin{align} \label{expSLploc}
       \sup_{n} \| \exp(-S(P_n(x))) \1_{B_R} \|_{L^p}(\mu) < \infty.
    \end{align}
\end{assumption}

\begin{assumption}[Derivative bounds]\label{ass6}
    For every $p >0$,
    \begin{equation} \label{Qin0}
        \sup_{n \in \N}\int_{B_R} \exp(p |\QQ_1^n|(x)) d \rho(x) = C_p(R) < \infty
    \end{equation}
    and for every $0 \le |t| \le T$,
    \begin{equation} \label{spacetimeQ}
        \sup_n \int_{B_R} \exp\Big(p \int_0^t \QQ_2^n(\Phi(\tau,x))\Big) d \rho(x) = D_p(R,T) < \infty.
    \end{equation}
\end{assumption}
\begin{proposition}\label{abstractquasiinvariance}
    Under Assumptions 1--6, for every $T \in \R $ we have that 
    $$ \Phi(T,\cdot)_{\#}\rho \ll \rho, $$
    and moreover, for every $p > 1, R >0,$
    \begin{equation}\label{localDensityBound}
        \int_{X} \Big|\frac{d\Phi(T,\cdot)_{\#}(\1_{B_R}\rho)}{d\rho}\Big|^p d\rho \le \big(C_{2Tp}(R'(T))\big)^{\frac{2p^2}{(2p-1)(2p-2)}} \big(D_{2(p-1)}(R'(T),T))^\frac{p}{2p-1} 
    \end{equation}
\end{proposition}
\begin{proof}
    We start by showing the analogous bounds for the measure 
    $$ \rho_n(x) = \rho(P_n(x)) = \frac1 {Z_n} \exp\Big( - S(x) - \frac 12 \|x\|_{H}^2 \Big) dx, $$
    where $dx$ denotes the (or any) Lebesgue measure on $E_n$. 
    By \eqref{norm_derivative}, for any smooth, compactly supported functional $F : E_n \to E_n$, we have that 
    \begin{align*}
        &\int F(\Phi_n(t,x)) d \rho_n \\
        &= \frac1 {Z_n} \int_{E_n} F(\Phi_n(t,x))\exp\Big( - S(x) - \frac 12 \|x\|_{H}^2 \Big) dx \\
        &= \frac1 {Z_n} \int_{E_n} F(x)\exp\Big( - S(\Phi_n(-t,x)) - \frac 12 \|\Phi_n(-t,x)\|_{H}^2 \Big) dx \\
        &= \frac1 {Z_n} \int_{E_n} F(x)\exp\Big(- S(x) - \frac 12 \|x\|_{H}^2 + \int_{-t}^0 (\QQ_1^n + \QQ_2^n)(\Phi_n(t',x)) dt'\Big) dx \\
        &= \int_{E_n} F(x) \exp\Big( \int_0^t (\QQ_1^n + \QQ_2^n)(\Phi_n(-t',x)) dt'\Big)  d \rho_n(x),
    \end{align*}
        where we used Assumption \ref{divfree} for the second equality, and the flow property and the fundamental theorem of calculus for the third equality.
        In particular, we obtain that 
        \begin{equation} \label{e:densityformula}
            \Phi_n(t,\cdot)_\# \rho_n = \exp\Big( \int_0^t (\QQ_1^n + \QQ_2^n)(\Phi_n(-t',x)) dt'\Big) \rho_n =: f_t \rho_n.
        \end{equation}
        Now fix $T, R >0$, and for $R'$ as in Assumption \ref{ass_global},  define the set 
        \begin{equation}
            S_t = \{ x \in E_n: \sup_{0 \le |t'| \le t} \|\Phi_n(t',x)\|_{X} \le R'(T)\}. 
        \end{equation}
        By Assumption \ref{ass_global}, we have that for every $0 \le |t| \le T,$
        \begin{equation} \label{StInclusionI}
            B_R \subseteq S_t \subseteq B_{R'(T)}. 
        \end{equation} 
        Moreover, it is immediate to check that for $|t| > |t'|$, we have 
        \begin{equation} \label{StInclusionII}
            S_{t} \subseteq S_{t'}.
        \end{equation}
        Fix $p>1$. By \eqref{e:densityformula}, H\"older's inequality, \eqref{StInclusionI}, \eqref{spacetimeQ}, Jensen's inequality for the exponential, \eqref{StInclusionII}, and \eqref{Qin0}, we have that 
        \begin{align*}
            &\int \1_{S_t}(\Phi_n(-t,x))f_t^p(x) d \rho_n(x) \\
            &\int \1_{S_t}(\Phi_n(-t,x))f_t^{p-1}(x) d \big(\Phi_n(t,\cdot)_{\#} \rho_n\big)(x)\\
            &= \int \1_{S_t}(x)f_t^{p-1}(\Phi_n(t,x)) d \rho_n(x)\\
            &= \int \1_{S_t}(x) \exp\Big( (p-1) \int_0^t (\QQ_1^n + \QQ_2^n)(\Phi_n(t',x)) dt'\Big) d\rho_n(x) \\
            &\le \sqrt{D_{2(p-1)}(R',T)} \sqrt{\int \1_{S_t}(x) \exp\Big( 2(p-1) \int_0^t \QQ_1^n(\Phi_n(t',x)) dt'\Big) d\rho_n(x)} \\
            &\le \sqrt{D_{2(p-1)}(R',T)} \sqrt{\frac1t \int_0^t \1_{S_t}(x) \exp\Big( 2t (p-1)  \QQ_1^n(\Phi_n(t',x)) \Big) d\rho_n(x)} dt' \\
            &= \sqrt{D_{2(p-1)}(R',T)}
            \sqrt{\frac1t \int_0^t \1_{S_t}(\Phi_n(-t',x)) \exp\Big( 2t (p-1)  \QQ_1^n(x) \Big) f_{t'}(x) d\rho_n(x)} dt' \\
            &\begin{multlined}
                \le \sqrt{D_{2(p-1)}(R',T)}
            \bigg(\frac1t \int_0^t \int \1_{S_{t'}}(\Phi_n(-t',x)) f_{t'}^p(x) d\rho_n(x) d t'\bigg)^\frac1{2p} \\
        \phantom{XXXXXXXXXXXXXXXXX}\times \bigg(\int_{B_{R'}}  \exp\Big( 2t p  \QQ_1^n(x) \Big) d \rho_n \bigg)^{\frac{p}{2(p-1)}}
            \end{multlined} \\
            &\le \big(C_{2Tp}(R')\big)^{\frac{p}{2(p-1)}} \big(D_{2(p-1)}(R',T))^\frac12 \sup_{0 \le |t| \le T} \Big(\int \1_{S_{t'}}(\Phi_n(-t',x)) f_{t'}^p(x) d\rho_n(x)\Big)^\frac{1}{2p}.
        \end{align*}
        Therefore, by taking supremums in the LHS, from \eqref{StInclusionI} we deduce that 
        \begin{align*}
           \int \1_{B_R}(\Phi_n(-t,x))f_t^p(x) d \rho_n(x) &\le \int \1_{S_t}(\Phi_n(-t,x))f_t^p(x) d \rho_n(x) \\
           &\le  \big(C_{2Tp}(R')\big)^{\frac{2p^2}{(2p-1)(2p-2)}} \big(D_{2(p-1)}(R',T))^\frac{p}{2p-1},  
        \end{align*}
           which is exactly \eqref{localDensityBound} but for the flows $\Phi_n$ and the measure $\rho_n$. Therefore, we just need to show that this bound passes to the limit. More precisely, in order to conclude, it is enough to show that for every bounded Borel functional $F: X\to \R$, 
           \begin{equation}\label{Lptoprove}
               \int F(\Phi(t,x)) d \rho(x) \le \big(C_{2Tp}(R')\big)^{\frac{2p^2}{(2p-1)(2p-2)}} \big(D_{2(p-1)}(R',T))^\frac{p}{2p-1} \|F\|_{L^{p'}(\rho)} ,
           \end{equation}
           where $p' = \frac{p}{p-1}$ is the H\"older conjugate of $p$. Since continuous functions with bounded support are dense in $L^{p'}(\rho + \Phi(t,\cdot)_\#\rho)$, without loss of generality we can assume that $F$ is continuous with bounded support. By Fatou and \eqref{Rconvergence}, we have that
           \begin{align*}
               & \int F(\Phi(t,x)) d \rho(x) \\
               &= \int \lim_{n \to \infty}  F(\Phi_n(t,P_n(x))) \exp(-S(P_n(x))) d  \mu(x) \\
               &\le \liminf_{n \to \infty} \int F(\Phi_n(t,x)) d \rho_n(x) \\
               &\le \big(C_{2Tp}(R')\big)^{\frac{2p^2}{(2p-1)(2p-2)}} \big(D_{2(p-1)}(R',T))^\frac{p}{2p-1} \liminf_{n\to \infty }\|F\|_{L^{p'}(\rho_n)}. 
           \end{align*}
           Since $F$ has bounded support on a ball $B_R$, by Assumption \ref{ass5} we have that $\exp(S\circ P_n) \1_{B_R}$ is equi-integrable, and by \eqref{Rconvergence} we have that $\exp(S\circ P_n) \to \exp(S)$ $\mu-a.e.$. Therefore, by equi-integrable convergence, we have that 
           $$ \liminf_{n\to \infty }\|F\|_{L^{p'}(\rho_n)} = \|F\|_{L^{p'}(\rho)}.$$
           In particular, we deduce \eqref{Lptoprove}, which concludes the proof.
\end{proof}

\section{The Local Cauchy Theory \texorpdfstring{in the space ${\mathcal F}L^{s,\infty}$}{Fourier-Lebesgue spaces}}

\subsection{The Cauchy theory in \texorpdfstring{$H^\sigma$, $\s\geq 1$}{Sobolev spaces}}

We recall the following Strichartz estimates proved in \cite{BGT} 
for the linear Schr\"odinger flow on $\T^2$ (indeed the estimate below is true for a general compact surface):
$$\|e^{it\Delta} u_0\|_{L^p_IL^q}\les \|u_0\|_{H^\frac 1p},$$
where $I\subset \R$ is a given bounded time interval and
$\frac 1p + \frac 1q=\frac 12, \quad p>2$.
In turn, by the Sobolev embedding 
$W^{1-\frac 1p, q}\subset 
L^\infty$ (notice that  $(1-\frac 1p)q>2$),
the estimate above implies 
\begin{equation}\label{cauchyBGT} \|e^{it\Delta} u_0\|_{L^p_IL^\infty}
\les \|u_0\|_{H^1}.\end{equation}
Along with this bound, we also get 
for the Duhamel operator 
\begin{equation}\label{duhamBGT}
\Big \|\int_0^t e^{i(t-\tau)\Delta} F(\tau) ds\tau \Big \|
_{L^p_IL^\infty} \les  \|F\|_{L^1_IH^1}\end{equation}
By combining \eqref{cauchyBGT} and
\eqref{duhamBGT} we can perform a fixed point argument in the space
$$Y_{[0,T]}^p=L^\infty_{[0,T]} H^1\cap L^p_{[0,T]}L^\infty,$$
see \cite{BGT} for details.
More precisely one can prove the following result.
\begin{proposition}\label{localH1}
There exists a universal constant $C>0$ such that, for every $R>1$ the Cauchy problem \eqref{NLS},
with $u_0\in H^1$ such that $\|u_0\|_{H^1}\leq R$, admits one unique local solution $u\in Y^{4k}_{[0,T]}$ with
$$\|u\|_{Y^{4k}_{[0,T]}}\leq C R$$ and
$T\sim CR^{-4k}$.
\end{proposition}
By the conservation of energy and the fact that \eqref{NLS} is defocusing, we can iterate the local result above and we get the following global one.
\begin{proposition}\label{cauchyl}
Let $T>0$ be given. Then there exists $C>0$ such that,
for any given $R>1$ the Cauchy problem 
\eqref{NLS}, with $u_0\in H^1$ such that $\|u_0\|_{H^1}\leq R$, admits one unique solution 
in the space
$u\in Y^{4k}_{[0,T]}$ such that $$\|u\|_{Y^{4k}_{[0,T]}}\leq C R^{1+4k},$$
where the constant $C$ depends only on $T$.
    
\end{proposition}

We shall need the following result on the propagation of regularity.

\begin{proposition} 
    Let $\s > 1$ be given. For every $T > 0$, $R> 1$, there exists $K>0$ such that for every $u_0 \in H^\s \cap H^1$
    with $\|u_0\|_{H^1}\leq R$, the Cauchy problem \eqref{NLS} admits one unique solution $u \in L^\infty_{[0,T]} H^\s$ such that 
    \begin{equation} \label{eq:H1+GWP}
        \| u \|_{L^\infty_{[0,T]} H^\s} \le K \|u_0\|_{H^\s}.
    \end{equation}
\end{proposition}

\begin{proof}
Since $H^\s$ is an algebra, one can show that the Cauchy problem \eqref{NLS} admits one unique local solution $u\in L^\infty_{[0, \tilde T]} H^\s$, for $u_0\in H^\sigma$ and $\tilde T>0$. Moreover, by recalling Proposition \ref{localH1}, we have that in fact 
$u\in L^\infty_{[0, \tilde T]} H^\s\cap Y_{[0, \tilde T]}^{4k}$. We aim at proving that, if we denote by $T_{max}$ the maximal time of existence, then necessarily $T_{max}\geq T$ and moreover we have the bound \eqref{eq:H1+GWP}. 
Let now assume by the absurd $T_{max}<T$, namely $\lim_{t\rightarrow T_{max}} \|u(t)\|_{H^\s}=\infty$.
For a given time interval 
$J=[a,b]\subset [0, T_{max})\subset [0, T)$ we get
$$\|u\|_{L^\infty_{J}H^\s}\leq
\|u(a)\|_{H^\s} + \|u\|_{L^\infty_{J}H^\s} \int_J \|u(\tau)\|^{2k}_{L^\infty}
\leq \|u(a)\|_{H^\s} + \|u\|_{L^\infty_{J}H^\s}
\|u\|^{2k}_{L^{4k}_{J}L^\infty} \sqrt{b-a}$$
and hence if we impose $(CR^{1+4k})^{2k}\sqrt{b-a}< \frac 12$, then by Proposition \ref{cauchyl} we get
$$\|u\|_{L^\infty_{J}H^\s}\leq 2 \|u(a)\|_{H^\s}$$
Next we split $[0, T_{max})=\cup_{1}^N I_j$, where $I_j=[\tau_j, \tau_{j+1})$ with 
$(CR^{1+4k})^{2k}\sqrt{|\tau_{j+1}-\tau_j|}<\frac 12$,
and we can repeat the argument above on each interval $I_j$: 
\begin{equation}\label{localized}
\|u(t)\|_{L^\infty_{I_j}H^\s}\leq 2 \|u(\tau_j)\|_{H^\s}.
\end{equation}
Notice that from the previous bound we get
$$\|u(\tau_{j+1})\|_{H^\s}
\leq 2 \|u(\tau_{j})\|_{H^\s}$$
and hence
$$\|u(\tau_{j})\|_{H^\s}\leq 2^j \|u_0\|_{H^\s},$$ which in conjunction with
\eqref{localized} implies
$$\|u(t)\|_{L^\infty_{[0, T_{max})}H^\s} \leq
2 \sup_j \|u(\tau_j)\|_{H^\s}
\leq 2^{N+1}\|u_0\|_{H^\s}.$$
As a consequence we have that $T_{max}\geq T$ and moreover we get the desired bound if we call $K=2^{N+1}$. Notice that by definition $N$ depends only on $R, T>0$.
\end{proof}

\subsection{The Cauchy theory in \texorpdfstring{$\FL^{\s,\infty}$, $\s>2$}{Fourier-Lebesgue for s>2}}\label{cauchyfl}
We will work extensively with the following reformulation of \eqref{NLS} 
in interaction representation, i.e.\ 
\begin{equation*}
    \partial_t v = -i e^{-it\Delta}\big(|e^{it\Delta } v|^{2k} e^{it\Delta } v\big),
\end{equation*}
or more specifically, with its mild formulation 
\begin{equation} \label{NLS_interaction}
    v(t) = v(0) - i \int_0^t e^{-i\tau \Delta}\big(|e^{i\tau \Delta } v(\tau)|^{2k} e^{i\tau\Delta } v(\tau)\big) d\tau.
\end{equation}
We will say that $u$ is a \emph{solution} of \eqref{NLS} if and only if $v := e^{-it\Delta} u$ solves \eqref{NLS_interaction}.\\
We denote for shortness
$X_\s={\mathcal F}L^{\s,\infty}$
and we also introduce 
the space $\mathcal A$
whose norm is defined as follows:
$$\|f\|_{\mathcal A}=\sum_{j\in \Z^2} |\hat f(j)|$$
One can easily show the following embeddings: 
$$X_{2+} \subset {\mathcal A}
\hbox{ and }  H^{1+}\subset {\mathcal A}.
$$
We need the following algebra property.
\begin{proposition}[Algebra property]\label{algebraprop}
Let $\s>2$ be given, then we have the following estimate
\begin{equation} \label{eq:algebra}
\|f g\|_{X_\s}\les \|f\|_{X_\s} \|g\|_{\mathcal A} 
+ C \|g\|_{X_\s} \|f\|_{\mathcal A} 
\end{equation}
for a suitable $C>0$.
\end{proposition}

\begin{proof}
We denote by $\{f_n\}_{n\in \Z^2}$ and $\{g_n\}_{n\in \Z^2}$ the fourier coefficients of $f, g$. 
We notice that
\begin{multline*}
\langle j\rangle^\s |\sum_{j=j_1+j_2} f_{j_1} g_{j_2}|
\leq \langle j\rangle^\s \sum_{\substack{j_1+j_2=j\\ \hbox{ s.t. } |j_1|\geq \frac{|j|}2}} 
|f_{j_1}| |g_{j_2}|+\langle j\rangle^\s |\sum_{\substack{j_1+ j_2=j\\|j_2|\geq \frac{|j|}2}} 
|f_{j_1}| |g_{j_2}|
\\\lesssim \sum_{j_1+j_2=j} 
\langle j_1\rangle^\s |f_{j_1}| |g_{j_2}|+  \sum_{j_1+ j_2=j} 
\langle j_2\rangle^\s |f_{j_1}| |g_{j_2}|\\
\lesssim\|f\|_{X_\s} \sum_{j_2}|g_{j_2}|+ \|g\|_{X_\s} \sum_{j_2}|f_{j_2}|. 
\end{multline*}
\end{proof}




\begin{corollary}
    Let $\s > 2$, $R>1$ and $T > 0$
    be fixed. Then for every $u_0 \in X_\s$ with $\|u_0\|_{H^1}\leq R$, there exists a unique solution $u(t) \in C([0,T], X_\s)$ of \eqref{NLS}. Moreover there exists a function
    $ F(\|u_0\|_{X_\s})$ such that
    \begin{equation}\label{eq:XsGWP}
        \| v \|_{C^1([0,T], X_\s)} \le F(\|u_0\|_{X_\s}).
    \end{equation}
\end{corollary}
\begin{proof} As a consequence of Proposition \ref{algebraprop} and the embedding $X_{\s} \subset \mathcal A$ we have that \eqref{NLS} is locally well posed in $X_\s$. Next we control 
$\|u\|_{L^\infty_{[0,T]}X_\s}$. Let $1< \s' < \s-1$ be fixed, then we have that $\|u_0\|_{\mathcal A} \les \|u_0\|_{H^{\s'}}$. 
    From 
    \eqref{eq:algebra} and \eqref{eq:H1+GWP}, we obtain that for some universal constant $C>0$,
    \begin{align*}
        \| u(t) \|_{X_\s} &\le \|u_0\|_{X_\s} + \int_0^t \| | u(\tau)|^{2k}  u(\tau) \|_{X_\s} d\tau \\
        &\le \|u_0\|_{X_\s} + C\int_0^t \|u(\tau)\|_{\mathcal{A}}^{2k} \|u(\tau)\|_{X_\s} d\tau \\
        &\le \|u_0\|_{X_\s} + C \int_0^t \|u(\tau)\|_{H^{\s'}}^{2k} \|u(\tau)\|_{X_\s}d\tau \\
        &\le \|u_0\|_{X_\s} + C K^{2k}\|u_0\|_{H^{\s'}}^{2k} \int_0^t \|u(\tau)\|_{X_\s} d\tau. 
    \end{align*}
    Therefore, we deduce from the Gronwall and the bound $\|u_0\|_{H^{\s'}} \les \|u_0\|_{X_\s}$ that the solution $u(t)$ can be extended up to time $T$ with the bound
    \begin{equation} \label{XsGWPC0}
        \| u \|_{C([0,T], X_\s)} \le F(\|u_0\|_{X_\s}).
    \end{equation} We now move to showing \eqref{eq:XsGWP}. By \eqref{NLS_interaction}, we have that 
    $$ 
    \partial_t v = - i e^{-it \Delta}\big(|e^{it \Delta } v(t)|^{2k} e^{it\Delta } v(t)\big) = -i e^{-it \Delta} \big(|u|^{2k}u).
    $$
    Therefore, \eqref{eq:XsGWP} follows from \eqref{XsGWPC0}, the embedding $X_\s \supset \mathcal{A}$ and \eqref{eq:algebra}.
\end{proof}

\begin{remark}
With the same proofs, the results above can be extended with uniform bounds 
to the solutions of the equation of the finite dimensional equations
$$i\partial_t u +\Delta u=P_{\le N} (P_{\le N}u|P_{\le N} u|^{2k}).$$
    
\end{remark}

\section{Modified energy and properties of the measure}\label{modifiedenergynew}

\subsection{Modified energy}
\begin{proposition}\label{modifiedenergy}
Given $\sigma\in \R$ let
$$\mathcal C^\s(u)= D^\s (u|u|^{2k}) - (k+1) D^\s u u^k \bar u^k - k u^{k+1} \bar u^{k-1} D^\s \bar u.$$
We have the following identity for every $s\geq 2$ and for $k\geq 1$ integer:
\begin{equation}
\begin{aligned}\label{energyestimate}
&\frac d{dt} \Big (\|D^s u\|_{L^2}^2 -
k \Re (D^{s-2} \bar u u^{k+1} \bar u ^{k-1}, D^s u)\Big)
\\&= (\mathcal C^s (u), D^s u) 
-2k (k-1)
\Im (D^{s-2} \bar u u^{k+1}  \Delta \bar u \bar u^{k-2}, D^s u)
\\&\phantom{=\ }- 2k(k+1) 
\Im (D^{s-2} \nabla \bar u \nabla u u^{k} \bar u^{k-1}, D^{s} u)
\\&\phantom{=\ }-  2k(k-1) 
\Im (D^{s-2} \nabla \bar u u^{k+1} \nabla \bar u \bar u^{k-2}, D^{s} u)
\\&\phantom{=\ }- 2 k (k+1)(k-1)\Im (D^{s-2} \bar u \nabla  u\nabla \bar u  u^{k} \bar u^{k-2}, D^{s} u)
- k {\mathcal N}(u),
\end{aligned}
\end{equation}
where
\begin{equation} \label{energy_nonlinearity}
    \begin{aligned}
    &{\mathcal N}(u)\\
&= 2 \Im (\nabla C^{s-2}(\bar u) |u|^{2k-2}u^2, \nabla D^{s-2} u)\\
&\phantom{=\ }+2k\Im (\nabla D^{s-2}\bar u \bar u^ku^{k} |u|^{2k-2}u^2, \nabla D^{s-2} u)\\
&\phantom{=\ }+(2k+2)\Im (D^{s-2}\bar u \nabla (|u|^{2k})|u|^{2k-2}u^2, \nabla D^{s-2} u)
\\&\phantom{=\ }+k
\Im (D^{s-2}u \nabla (|u|^{2k-2} \bar u^{2}) |u|^{2k-2}u^2, \nabla D^{s-2} u)
\\&\phantom{=\ }+
\Im (D^{s-2}(\bar u|u|^{2k} ) \nabla (|u|^{2k-2}u^2), \nabla D^{s-2} u)
\\&\phantom{=\ }-2\Im (D^{s-2} \bar u \nabla(|u|^{4k-2} u^2), \nabla D^{s-2}u)
\\&\phantom{=\ } +\Im (D^{s-2}\bar u \nabla (|u|^{2k-2} u^2), \nabla D^{s-2} (u|u|^{2k}))\\
&\phantom{=\ } +k\Im (\nabla D^{s-2}\bar u |u|^{2k-2} u^2, D^{s-2} \bar u \nabla(u^{k+1}\bar u^{k-1})).
\end{aligned}
\end{equation}
\end{proposition}

\begin{proof}
We compute
\begin{align*}
&\frac d{dt} \|D^s u\|_{L^2}^2=
2\Re (D^s \partial_t u, D^s u)
\\&=2\Re (-i D^s (u|u|^{2k}), D^s u)=2\Im (D^s (u|u|^{2k}), D^s u)
\\&=2 (k+1) \Im (D^s u |u|^{2k}, D^s u) + 
2 k \Im (D^s \bar u u^{k+1}\bar u^{k-1}, D^s u)+ 
(\mathcal C^s u, D^s u)\\&
= 2 k \Im (D^s \bar u u^{k+1}\bar u^{k-1}, D^s u)
+(\mathcal C^s u, D^s u).
\end{align*}
Next we introduce the following correction 
of the energy, in order to cancel the term 
$2 k \Im (D^s \bar u u^{k+1}\bar u^{k-1}, D^s u)$ on the r.h.s.:
\begin{align*}
&\frac d{dt} \Re (D^{s-2} \bar u u^{k+1} \bar u ^{k-1}, D^s u)\\&=
\Re (D^{s-2} \partial_t \bar u u^{k+1} \bar u ^{k-1}, D^s u)
+(k+1) \Re (D^{s-2} \bar u \partial_t u u^{k} \bar u ^{k-1}, D^s u)
\\&\phantom{=\ }
+(k-1) \Re (D^{s-2} \bar u u^{k+1} \partial_t \bar u \bar u^{k-2}, D^s u)
+\Re (D^{s-2} \bar u u^{k+1} \bar u ^{k-1}, D^s \partial_t u)
\\&={\mathcal L}(u) + \mathcal N(u)
\end{align*}
where ${\mathcal L}(u)$ is obtained by replacing $\partial_t u$ by $i\Delta u$
and ${\mathcal N}(u)$ is obtained by replacing $\partial_t u$ by $-iu|u|^{2k}$:
\begin{align*}
{\mathcal L}(u)&=\Im (D^{s} \bar u u^{k+1} \bar u ^{k-1}, D^s u)
- (k+1)\Im (D^{s-2} \bar u D^2 u u^{k} \bar u ^{k-1}, D^s u)
\\& \phantom{=\ }+ (k-1)
\Im (D^{s-2} \bar u u^{k+1}  D^2\bar u \bar u^{k-2}, D^s u)
+\Im (D^{s-2} \bar u u^{k+1} \bar u ^{k-1}, D^{s+2} u)\\
&=
\Im (D^{s} \bar u u^{k+1} \bar u ^{k-1}, D^s u)
- (k+1)\Im (D^{s-2} \bar u D^2 u u^{k} \bar u ^{k-1}, D^s u)
\\&\phantom{=\ }+ (k-1)
\Im (D^{s-2} \bar u u^{k+1}  D^2\bar u \bar u^{k-2}, D^s u)
+\Im (\Delta(D^{s-2} \bar u u^{k+1} \bar u ^{k-1}), D^{s} u)\\
&=\Im (D^{s} \bar u u^{k+1} \bar u ^{k-1}, D^s u)
- (k+1)\Im (D^{s-2} \bar u D^2 u u^{k} \bar u ^{k-1}, D^s u)
\\&\phantom{=\ }+ (k-1)
\Im (D^{s-2} \bar u u^{k+1}  D^2\bar u \bar u^{k-2}, D^s u)
\\&\phantom{=\ }+\Im (D^{s} \bar u u^{k+1} \bar u^{k-1}, D^{s} u)
+(k+1) \Im (D^{s-2} \bar u \Delta u u^{k} \bar u^{k-1}, D^{s} u)
\\&\phantom{=\ }+(k-1)\Im (D^{s-2} \bar u u^{k+1} \Delta \bar u \bar u^{k-2}, D^{s} u)
\\&\phantom{=\ }+ 2(k+1) 
\Im (D^{s-2} \nabla \bar u \nabla u u^{k} \bar u^{k-1}, D^{s} u)
+  2(k-1) 
\Im (D^{s-2} \nabla \bar u u^{k+1} \nabla \bar u \bar u^{k-2}, D^{s} u)
\\&\phantom{=\ }+ 2(k+1)(k-1)\Im (D^{s-2} \bar u \nabla  u\nabla \bar u  u^{k} \bar u^{k-2}, D^{s} u)\\
&= 2 \Im (D^{s} \bar u u^{k+1} \bar u ^{k-1}, D^s u)
\\&\phantom{=\ }+ 2(k-1)
\Im (D^{s-2} \bar u u^{k+1}  \Delta\bar u \bar u^{k-2}, D^s u)
\\&\phantom{=\ }+ 2(k+1) 
\Im (D^{s-2} \nabla \bar u \nabla u u^{k} \bar u^{k-1}, D^{s} u)
+  2(k-1) 
\Im (D^{s-2} \nabla \bar u u^{k+1} \nabla \bar u \bar u^{k-2}, D^{s} u)
\\&\phantom{=\ }+ 2(k+1)(k-1)\Im (D^{s-2} \bar u \nabla  u\nabla \bar u  u^{k} \bar u^{k-2}, D^{s} u)
\end{align*}
and
\begin{align*}
{\mathcal N}(u)
&=-\Im (D^{s-2} (\bar u |u|^{2k})u^{k+1} \bar u ^{k-1}, D^s u)
+(k+1) \Im (D^{s-2} \bar u u|u|^{2k} u^{k} \bar u ^{k-1}, D^s u)
\\
&\phantom{=\ }-(k-1) \Im (D^{s-2} \bar u u^{k+1} \bar u |u|^{2k}\bar u^{k-2}, D^s u)
-\Im (D^{s-2} \bar u u^{k+1} \bar u ^{k-1}, D^s (u|u|^{2k})) 
\\
&= -\Im (D^{s-2} (\bar u |u|^{2k})|u|^{2k-2}u^2, D^s u) + 2 \Im (D^{s-2} \cj u |u|^{4k-2}u^2, D^s u) \\
&\phantom{=\ }-\Im (D^{s-2}\cj u |u|^{2k-2}u^2, D^s (u|u|^{2k}))
\end{align*}
We develop the terms on the r.h.s.
\begin{align*}
&-\Im (D^{s-2}(\bar u|u|^{2k} )|u|^{2k-2}u^2, D^s u)\\
&=\Im (\nabla (D^{s-2}(\bar u|u|^{2k} )|u|^{2k-2}u^2, \nabla D^{s-2} u)+
\Im (D^{s-2}(\bar u|u|^{2k} ) \nabla (|u|^{2k-2}u^2), \nabla D^{s-2} u)\\
&= \Im (\nabla C^{s-2}(\bar u) |u|^{2k-2}u^2, \nabla D^{s-2} u)
+(k+1)\Im (\nabla (D^{s-2}\bar u \bar u^ku^{k} )|u|^{2k-2}u^2, \nabla D^{s-2} u)\\
&\phantom{=\ }+k
\Im (\nabla (D^{s-2}u u^{k-1} \bar u ^{k+1}) |u|^{2k-2}u^2, \nabla D^{s-2} u)
\\&\phantom{=\ }+
\Im (D^{s-2}(\bar u|u|^{2k} ) \nabla (|u|^{2k-2}u^2), \nabla D^{s-2} u)
\\
&= \Im (\nabla C^{s-2}(\bar u) |u|^{2k-2}u^2, \nabla D^{s-2} u)\\
&\phantom{=\ }+(k+1)\Im (\nabla D^{s-2}\bar u \bar u^ku^{k} |u|^{2k-2}u^2, \nabla D^{s-2} u)\\
&\phantom{=\ }+(k+1)\Im (D^{s-2}\bar u \nabla (|u|^{2k})|u|^{2k-2}u^2, \nabla D^{s-2} u)
\\&\phantom{=\ }+k
\Im (D^{s-2}u \nabla (|u|^{2k-2} \bar u^{2}) |u|^{2k-2}u^2, \nabla D^{s-2} u)
\\&\phantom{=\ }+
\Im (D^{s-2}(\bar u|u|^{2k} ) \nabla (|u|^{2k-2}u^2), \nabla D^{s-2} u)
\end{align*}
where used that
$$\Im (\nabla D^{s-2}u u^{k-1} \bar u ^{k+1} |u|^{2k-2}u^2, \nabla D^{s-2} u)=0.$$
Next we have 
\begin{align*}
&2\Im (D^{s-2} \bar u |u|^{4k-2} u^2, D^su)\\
&=-2\Im (\nabla (D^{s-2} \bar u |u|^{4k-2} u^2), \nabla D^{s-2}u)
\\&
=-2\Im (\nabla D^{s-2} \bar u |u|^{4k-2} u^2, \nabla D^{s-2}u)
-2\Im (D^{s-2} \bar u \nabla(|u|^{4k-2} u^2), \nabla D^{s-2}u)
\end{align*}
and we also get
\begin{align*}
&-\Im (D^{s-2}\bar u |u|^{2k-2} u^2, D^s (u|u|^{2k}))\\
&=\Im (D^{s-2}\bar u \nabla (|u|^{2k-2} u^2), \nabla D^{s-2} (u|u|^{2k}))+\Im (\nabla D^{s-2}\bar u |u|^{2k-2} u^2, \nabla D^{s-2} (u|u|^{2k}))
\\
&=
\Im (D^{s-2}\bar u \nabla (|u|^{2k-2} u^2), \nabla D^{s-2} (u|u|^{2k}))
+\Im (\nabla D^{s-2}\bar u |u|^{2k-2} u^2, \nabla C^{s-2} (u))\\
&\phantom{=\ }+(k+1)\Im (\nabla D^{s-2}\bar u |u|^{2k-2} u^2, \nabla (D^{s-2} u u^{k}\bar u^{k}))\\
&\phantom{=\ }+k\Im (\nabla D^{s-2}\bar u |u|^{2k-2} u^2, \nabla (D^{s-2} \bar u u^{k+1}\bar u^{k-1}))
\\&=
\Im (D^{s-2}\bar u \nabla (|u|^{2k-2} u^2), \nabla D^{s-2} (u|u|^{2k}))\\
&\phantom{=\ }+\Im (\nabla D^{s-2}\bar u |u|^{2k-2} u^2, \nabla C^{s-2} (u))\\
&\phantom{=\ }+(k+1)\Im (\nabla D^{s-2}\bar u |u|^{2k-2} u^2, \nabla D^{s-2} u u^{k}\bar u^{k})\\
&\phantom{=\ }+(k+1)\Im (\nabla D^{s-2}\bar u |u|^{2k-2} u^2, D^{s-2} u \nabla (u^{k}\bar u^{k}))\\
&\phantom{=\ } +k\Im (\nabla D^{s-2}\bar u |u|^{2k-2} u^2, D^{s-2} \bar u \nabla(u^{k+1}\bar u^{k-1}))
\end{align*}
where we used that
$$\Im (\nabla D^{s-2}\bar u |u|^{2k-2} u^2, \nabla D^{s-2} \bar u (u^{k+1}\bar u^{k-1}))=0.$$
Therefore we have
\begin{align*}
&{\mathcal N}(u)\\
&= 2 \Im (\nabla C^{s-2}(\bar u) |u|^{2k-2}u^2, \nabla D^{s-2} u)\\
&\phantom{=\ }+2k\Im (\nabla D^{s-2}\bar u \bar u^ku^{k} |u|^{2k-2}u^2, \nabla D^{s-2} u)\\
&\phantom{=\ }+(2k+2)\Im (D^{s-2}\bar u \nabla (|u|^{2k})|u|^{2k-2}u^2, \nabla D^{s-2} u)
\\&\phantom{=\ }+k
\Im (D^{s-2}u \nabla (|u|^{2k-2} \bar u^{2}) |u|^{2k-2}u^2, \nabla D^{s-2} u)
\\&\phantom{=\ }+
\Im (D^{s-2}(\bar u|u|^{2k} ) \nabla (|u|^{2k-2}u^2), \nabla D^{s-2} u)
\\&\phantom{=\ }-2\Im (D^{s-2} \bar u \nabla(|u|^{4k-2} u^2), \nabla D^{s-2}u)
\\&\phantom{=\ } +\Im (D^{s-2}\bar u \nabla (|u|^{2k-2} u^2), \nabla D^{s-2} (u|u|^{2k}))\\
&\phantom{=\ } +k\Im (\nabla D^{s-2}\bar u |u|^{2k-2} u^2, D^{s-2} \bar u \nabla(u^{k+1}\bar u^{k-1})).
\end{align*}
\end{proof}

\subsection{Construction of the measure}

We define 
\begin{equation}\label{Sdef}
S(u)=-\frac k2\lim_{N\rightarrow \infty}
\Re (D^{s-2} P_{\le N} \bar u (P_{\le N} u)^{k+1} (P_{\le N}\bar u) ^{k-1}, D^s P_{\le N}u)
\end{equation}
whenever the limit exists and $0$ otherwise.
We now show that $S$ satisfies Assumption \ref{ass6},
with $P_n=P_{\le N}$.  
\begin{lemma}
We have
\begin{equation} \label{conj_variance}
\E \big|(D^s P_{N} \bar u D^{s-2}  P_{M} \bar u)\ft{\phantom X} (n)\big|^2 \les \min(M^{-2}, \1_{|n| \sim N} + \jb{n}^{-2}\1_{N \ll |n|} + N^{-2} \1_{N \gg |n|}), 
\end{equation}
where expectations are taken with respect to $\mu$.
In particular, for every $2 < \s < s$,
\begin{equation} \label{conj_exponential}
    \big(\E \| (D^s P_{N} \bar u D^{s-2}  P_{M} \bar u) \|_{\mathcal FL^{-\s,1}}^p\big)^{\frac1p} \les p \min(M^{-1}, N^{-1}).
\end{equation}
\end{lemma}
\begin{proof}
We have 
\begin{equation}
(D^s P_{N} \bar u D^{s-2}  P_{M} \bar u)
\ft{\phantom X} (n)=\sum_{\substack{l+k=n\\|l|\sim N, |k|\sim M}}|l|^s|k|^{s-2}\frac{\cj {g_l}}{\langle l \rangle^s} \frac{\cj {g_k}}{\langle k \rangle^s}.
\end{equation}
Hence we get 
\begin{equation}\E \big|(D^s P_{N} \bar u D^{s-2}  P_{M} \bar u)\ft{\phantom X} (n)\big|^2
\les \sum_{\substack{l+k=n\\|l|\sim N, |k|\sim M}}
\frac 1{{\jb k}^4}.
\end{equation}
Notice that summing over $k$ we get
\begin{equation*}
\E \big|(D^s P_{N} \bar u D^{s-2}  P_{M} \bar u)\ft{\phantom X} (n)\big|^2
\les M^{-2}.
\end{equation*}
On the other hand summing over $l$ we get
\begin{align*}
    \E \big|(D^s P_{N} \bar u D^{s-2}  P_{M} \bar u)\ft{\phantom X} (n)\big|^2
&\les \sum_{\substack{l+k=n\\|l|\sim N, |k|\sim M}}
\frac 1{{\jb k}^4} \\
& \les \sum_{|l| \sim N} \frac{1}{\jb{n-l}^4} \\
&\les \1_{N \ll |n|}\sum_{|l| \sim N} \frac{1}{\jb{n}^4} +  \1_{N \sim |n|}\sum_{|l| \sim N} \frac{1}{\jb{n-l}^4} + \1_{N \gg |n|}\sum_{|l| \sim N} \frac{1}{N^4} \\
&\les \1_{N \ll |n|} \frac{N^2}{\jb{n}^4} + \1_{N \sim |n|} + \1_{N \gg |n|} N^{-2},
\end{align*}
which concludes the proof of \eqref{conj_variance}. Since the expression in \eqref{conj_variance} is a polynomial of degree 2 in a Gaussian random variable, by standard Wiener Chaos estimates (see \cite[Theorem I.22]{Simon}), we deduce that 
\begin{equation}
    \Big(\E \big|(D^s P_{N} \bar u D^{s-2}  P_{M} \bar u)\ft{\phantom X} (n)\big|^p\Big)^\frac1p \les p \min(M^{-1}, \1_{|n| \sim N} + \jb{n}^{-1}\1_{N \ll |n|} + N^{-1} \1_{N \gg |n|}),
\end{equation}
where the implicit constant is independent of $p$. In particular, by Minkowski, 
\begin{align*}
    & \big(\E \| (D^s P_{N} \bar u D^{s-2}  P_{M} \bar u) \|_{\mathcal FL^{-\s,1}}^p\big)^{\frac1p} \\
    &= \Big(\E \Big|\sum_{n \in \Z^2} \jb{n}^{-\s} \big|(D^s P_{N} \bar u D^{s-2}  P_{M} \bar u)\ft{\phantom X} (n)\big|\Big|^p \Big)^{\frac1p} \\
    &\le \sum_{n \in \Z^2} \jb{n}^{-\s} \Big(\E \big|(D^s P_{N} \bar u D^{s-2}  P_{M} \bar u)\ft{\phantom X} (n)\big|^p\Big)^\frac1p \\
    &\les p \sum_{n \in \Z^2} \jb{n}^{-\s} \min(M^{-1}, \1_{|n| \sim N} + \jb{n}^{-1}\1_{N \ll |n|} + N^{-1} \1_{N \gg |n|}) \\
    &\les p \min(M^{-1}, N^{2-\s} + N^{1-\s}, N^{-1}) \\
    &\les p \min(M^{-1}, N^{-1}).
\end{align*}
\end{proof}

\begin{lemma} \label{lem:Sconvergence}
For every $2<\sigma<s$, the limit \eqref{Sdef} exists for $\mu$ a.e. $u\in {\mathcal F}L^{\s,\infty}$.    
\end{lemma}

\begin{proof}
Let $B_R$ be the ball 
$$ B_R = \{ u \in \FL^{\s, \infty}: \|u\|_{\FL^{\s, \infty}} \le R\}. $$
We introduce the notation
\begin{equation}\label{Bqdef}
    q(u)=|u|^{2k-2}u^2, \quad B(u) = (D^s P_{N} \bar u D^{s-2}  P_{M} \bar u).
\end{equation}
Note that with this notation, we have that 
$$ S(P_{\le N} u) = \int_{\T^2 }q(P_{\le N} u) B(P_{\le N} u).$$
Therefore, it is enough to show that for $M>N$ dyadic,
\begin{equation} \label{to_show_conv}
    \E  [\1_{B_R}(u) \big|\int_{\T^2} \big(q(P_{\le M} u) B(P_{\le M} u) - q(P_{\le N} u) B(P_{\le N} u)\big) \big|^2]
\les_R N^{0-}. 
\end{equation} 
For $2< \s' < \s $, from the algebra property we get that 
\begin{equation}\label{algebra_smallness}
    \|q(P_{\leq N}u) - q(P_{\leq M} u)\|_{{\mathcal F}L^{\s',\infty}}\les \|u\|^{2k-1}_{{\mathcal F}L^{\s',\infty}} \|P_{\leq M} u - P_{\leq N} u\|_{{\mathcal F}L^{\s',\infty}}\les N^{-(\s-\s')}
R^{2k}.
\end{equation}
From \eqref{conj_exponential}, by expanding $B$ in dyadic pieces, we have that 
\begin{equation} \label{conj_exponential_2}
       \E \| B(P_{\le N} u) \|_{\FL^{-\s,1}}^2 \les 1, \quad \E \| B(P_{\le N} u) - B(P_{\le M} u) \|_{\FL^{-\s,1}}^2 \les N^{-1}.
\end{equation}
Therefore, from \eqref{algebra_smallness} and \eqref{conj_exponential_2}, we obtain 
\begin{equation} \label{3epsBq}
\begin{aligned}
&\E  [\1_{B_R}(u) \big|\int_{\T^2} \big(q(P_{\le M} u) B(P_{\le M} u) - q(P_{\le N} u) B(P_{\le N} u)\big) \big|^2] \\
& 
\les \E [\1_{B_R}(u) \| B(P_{\le N} u) \|_{\FL^{-\s',1}} \| q(P_{\le N} u) - q(P_{\le M} u)\|_{\FL^{\s',\infty}}] \\
&\phantom{\les\ 
}+ \E [\1_{B_R}(u) \| B(P_{\le N} u) - B(P_{\le M}u)\|_{\FL^{-\s',1}} \| q(P_{\le M} u)\|_{\FL^{\s',\infty}}]\\
&\les N^{-(\s-\s')}R^{2k} + N^{-1}R^{2k},
\end{aligned} 
\end{equation}

which shows \eqref{to_show_conv}.
\end{proof}
\begin{lemma} \label{lem:Sexploc}
    Let $R > 0$, and let $B_R$ be the ball 
    $$ B_R = \{ u \in \FL^{\s, \infty}: \|u\|_{\FL^{\s, \infty}} \le R\}. $$ 
    For every $1 < p < \infty$, we have that   
    \begin{equation} \label{Smomentbound}
        \int_{B_R} \exp(p |S(u)|) d \mu(u) < \infty. 
    \end{equation}
\end{lemma}
\begin{proof}
    Note that, for every $N_0$ dyadic, we have that 
    $$ |S(P_{N_0} u)| \les \| u \|_{\FL^{\s,\infty}}^{2k} N_{0}^{2s-2}. $$
    Therefore, it is enough to show that there exists $N_0$ dyadic such that 
        \begin{equation}
        \int_{B_R} \exp(p |S(u) - S(P_{N_0}(u))|) d \mu(u) < \infty. 
    \end{equation}
    Recall the notation for $B(u), q(u)$ in \eqref{Bqdef}. By \eqref{conj_exponential}, for every $2 < \s' < \s$, we obtain that 
    \begin{equation}
        \E \| B(u) - B(P_{\le N_0} u)\|_{\FL^{-\s', 1}}^p \le C^p N_{0}^{-p}, 
    \end{equation}
    for some constant $C = C(\s')$. Therefore, proceeding as in \eqref{3epsBq}, we obtain that 
    \begin{equation}
        \E \| S(u) - S(P_{\le N_0} u)\|_{\FL^{-\s', 1}}^p \le C^p N_{0}^{-p \min(1, \s - \s')},
    \end{equation}
    for a (different) constant $C = C(\s, \s')$. From this, we deduce that 
    \begin{align*}
        &\int_{B_R} \exp(p |S(u) - S(P_{N_0}(u))|) d \mu(u) \\
        &= \sum_{k=0}^\infty \int_{B_R} \frac{p^k |S(u) - S(P_{N_0}(u))|^k}{k!} d\mu(u)\\
        &\le \sum_{k=0}^\infty \frac{C^k N_{0}^{-k \min(1, \s - \s')}}{k!}.
    \end{align*}
    Therefore, by picking $N_0 \gg 1$ such that 
    $$ C N_0^{-\min(1, \s - \s')} < e,$$
    the series converges, and we complete the proof.
\end{proof}

\section{Multilinear estimates along the flow} \label{sec:butcherdeng}
\begin{definition} \label{XTspace}
    Let $\sigma < s$, $T \ge 0$. Define the space $X_T$ given by the space-time functions such that the following norm is finite: 
    \begin{equation}
        \| f \|_{X_T} = \sup_{n \in \Z^2, t\in [0,T]} \jb{n}^\sigma |\ft{f}(t,n)| + \sup_{n \in \Z^2, t\in [0,T]} \jb{n}^\sigma \big|\partial_t \big(e^{it|n|^2}\ft{f}(t,n)\big)\big|
    \end{equation}
    Define $\cj X_T$ to be the space 
    $$ \cj X_T = \{f: \cj f \in X_T\},  $$
    with norm
    $$ \| f\|_{\cj X_T} := \| \cj f\|_{X_T}. $$
\end{definition}
\subsection{Counting estimate bounds}

Recall the following counting estimate from \cite[Lemma 4.5]{DNY3}.
\begin{proposition}
    Let $m\in \Z^2, \kappa \in \Z$, and let $N_1,N_2,N_3 \ge 1$ be dyadic numbers. Consider the set 
    \begin{align*}
        &S:= S_{m,\kappa} := S_{m,\kappa}(N_1,N_2,N_3) \\
        &:= \{ (k_1,k_2,k_3): \pm k_1 \pm k_2 \pm k_3 = m, \pm k_1^2 \pm k_2^2 \pm k_3^2 = \kappa, k_j \sim N_j\} \\
        &\phantom{XX}\setminus \bigcup_{j=1,2,3} \{  (k_1,k_2,k_3): k_j \pm k_{j+1} = 0\}.  
    \end{align*}
    Then 
    \begin{equation} \label{counting}
       |S_{m,\kappa}(N_1,N_2,N_3)| \les N_{(2)}^{1+} N_{(3)}^{1+},
    \end{equation}
    where $N_{(1)} \ge N_{(2)}\ge N_{(3)}$ is a rearrangement of $N_1,N_2,N_3$, and the implicit constants are independent of $m, \kappa$.
\end{proposition}
\begin{proposition}
    Let $f_1, f_2, \dotsc f_{k+1} \in X_T$, $ g_{1},\dotsc, g_{k+1} \in \cj X_T$. Let $N_1,\dotsc,N_{k+1}$, $M_1,\dotsc,M_{k+1}$ be dyadic numbers. Let 
    $$ N_{(1)} \ge N_{(2)} \ge \dotsb N_{(2k+2)} $$
    be a rearrangement of $N_1,\dotsc,N_{k+1}, M_1,\dotsc,M_{k+1},$ let $n_{(1)}, \dotsc, n_{(2k+2)}$ be the corresponding rearrangement of  $n_1,\dotsc, n_{k+1}, m_{1}, \dotsc, m_{k+1}$, and let $f_{(1)}, \dotsc, f_{(2k+2)}$ be the corresponding rearrangement of $f_1, f_2, \dotsc f_{k+1}, g_{1},\dotsc, g_{k+1}$.
    Then for $0\leq t \leq T$
    \begin{equation}
        \begin{aligned}
        &\Big|\int_0^t\int  \prod_{j=1}^{k+1} P_{N_j}f_j(\tau) P_{M_j}g_j(\tau) d\tau - \int_0^T \Big(\int P_{N_{(1)}} f_{(1)}(\tau) P_{N_{(2)}} f_{(2)}(\tau)\Big) \int \prod_{j=3}^{2k+2} P_{N_{(j)}}f_{(j)}(\tau) d\tau  \Big| \\
        &\les_T N_{(1)}^{1-2\s+} N_{(3)}^{1-\s+} \prod_{j=4}^{2k+2} N_{(j)}^{2-\s} \prod_{j=1}^k \| f_j\|_{X_T}\| g_j\|_{\cj X_T}.
    \end{aligned} \label{butcherdeng}
    \end{equation}
    Moreover, 
    \begin{equation}
        \begin{aligned}
        &\Big|\int_0^t \Big(\int P_{N_{(1)}} f_{(1)}(\tau) P_{N_{(2)}} f_{(2)}(\tau)\Big) \int \prod_{j=3}^{2k+2} P_{N_{(j)}}f_{(j)}(\tau) d\tau  \Big| \\
        &\les_T N_{(1)}^{2-2\s+} N_{(3)}^{-\s} \prod_{j=4}^{2k+2} N_{(j)}^{2-\s} \prod_{j=1}^k \| f_j\|_{X_T}\| g_j\|_{\cj X_T}.
    \end{aligned} \label{butcherdengpairing}
    \end{equation}
\end{proposition}
\begin{proof}
    Let $F_j = e^{it\Delta} f_j$, 
    $G_j = e^{-it\Delta} g_j$. We have that 
    $$ \|F_j\|_{C^1([0,T], {\mathcal F}L^{\s,\infty})} \les \|f_j\|_{X_T},$$
    and similarly for $G_j$. Therefore, writing the LHS of \eqref{butcherdeng} in Fourier variables, and integrating by parts, we have that 
    \begin{align*}
        &\int_0^t\int  \prod_{j=1}^{k+1} P_{N_j}f_j(\tau) P_{M_j}g_j(\tau) d\tau - \int_0^t \Big(\int P_{N_{(1)}} f_{(1)}(\tau) P_{N_{(2)}} f_{(2)}(\tau)\Big) \int \prod_{j=3}^{2k+2} P_{N_{(j)}}f_{(j)}(\tau) d\tau \\
        &= \int_0^t \sum_{\substack{n_1 + \dotsb n_{k+1} + m_1 + \dotsb + m_{k+1} =0 \\ n_{(1)} \pm n_{(2)} \neq 0 \\
        n_j \sim N_j, m_j \sim M_j}} e^{i\tau \sum_{j=1}^{k+1} |n_j|^2 - |m_j|^2}  \prod_{j=1}^{k+1}\ft{F_j}(\tau,n_j)   \ft{G_j}(\tau, m_j)  d\tau \\
        & =\sum_{|\kappa| \les N_{(1)}^2}\sum_{E_\kappa} \int_0^t e^{i\tau \kappa} \prod_{j=1}^{k+1}\ft{F_j}(\tau, n_j) \ft{G_j}(\tau, m_j)  d\tau \\
        &= \left[\sum_{|\kappa| \les N_{(1)}^2}\sum_{E_\kappa} \frac{e^{i\tau \kappa} -1}{i\kappa} \prod_{j=1}^{k+1}\ft{F_j}(\tau, n_j) \ft{G_j}(\tau, m_j)  d\tau\right]_0^t \\
        &\phantom{=} - \sum_{|\kappa| \les N_{(1)}^2}\sum_{E_\kappa} \int_0^t \frac{e^{i\tau \kappa} -1}{i\kappa} \sum_{j_0 = 1}^{k+1} \partial_t(\ft{F_{j_0}}(\tau, n_j) \ft{G_{j_0}}(\tau, m_j)) \prod_{j\neq j_0}\ft{F_j}(\tau, n_j) \ft{G_j}(\tau, m_j)  d\tau,
    \end{align*}
    where 
    \begin{equation*}
            \begin{multlined}
        E_\kappa:= \{ (n_1, \dotsc, n_{k+1}, m_1, \dotsc, m_{k+1}): 
        n_j \sim N_j, m_j \sim M_j, n_{(1)} \pm n_{(2)} \neq 0\\ \sum_{j=1}^{k+1} n_j + m_j = 0, \sum_{j=1}^{k+1} |n_j|^2 - |m_j|^2 = \kappa\}.
    \end{multlined}
    \end{equation*}
    Therefore, we obtain 
    \begin{align*}
        &\bigg|\int_0^t\int  \prod_{j=1}^{k+1} P_{N_j}f_j(\tau) P_{M_j}g_j(\tau) d\tau - \int_0^t \Big(\int P_{N_{(1)}} f_{(1)}(\tau) P_{N_{(2)}} f_{(2)}(\tau)\Big) \int \prod_{j=3}^{2k+2} P_{N_{(j)}}f_{(j)}(\tau) d\tau\bigg| \\
        &\les_T \sum_{|\kappa| \les N_{(1)}}\sum_{E_\kappa} \frac{1}{\jb{\kappa}}\prod_{j=1}^{k+1} N_j^{-\s}\|F_j\|_{C^1([0,T], {\mathcal F}L^{\s,\infty})} M_j^{-\s}\|G_j\|_{C^1([0,T], {\mathcal F}L^{\s,\infty})} \\
        &\les \prod_{j=1}^{2k+2} N_{(j)}^{-\s}\sum_{|\kappa| \les N_{(1)}} \frac{1}{\jb{\kappa}} |E_\kappa| \prod_{j=1}^{k+1} \|f_j\|_{X_T} \|g_j\|_{\cj X_T} \\
        & \les N_{(1)}^{0+}\sup_{|\kappa| \les N_{(1)}} | E_\kappa| \prod_{j=1}^{2k+2} N_{(j)}^{-\s} \prod_{j=1}^{k+1} \|f_j\|_{X_T} \|g_j\|_{\cj X_T}. 
    \end{align*}
Recall that, in order to have the condition $n_1+\dotsb+n_{k+1}+m_1+\dotsb+m_{k+1}=0,$
we must have that $N_{(1)} \sim N_{(2)}$.
Therefore, to complete the proof of \eqref{butcherdeng}, it is enough to show that 
\begin{equation} \label{Ekappacounting}
    |E_\kappa| \les N_{(1)}^{1+} N_{(3)}^{1+} \prod_{j=4}^{2k+2} N_{(4)}^{2}.
\end{equation}  
We have that 
\begin{align*}
    E_{\kappa} &\subseteq \bigcup_{n_{(4)}, \dotsc, n_{(2k+2)}} S_{\sum_{j=4}^{2k+2} \pm n_{(j)}, \kappa - \sum_{j=4}^{2k+2} \pm |n_{(j)}|^2}(N_{(1)}, N_{(2)}, N_{(3)}) \\
    &\phantom{X}\cup \bigcup_{l \in \{1,2\}} \Big\{ n_{(j)} \sim N_j, n_{(l)} = \pm n_{(3)}, \sum_{j\neq l,3} \pm n_{(j)} =0 \Big\}.
\end{align*}
In order for the sets in the second line to be non-empty, we must have that  $N_{(1)} \sim N_{(3)}$, which leads to 
\begin{align*}
    \Big|\bigcup_{l \in \{1,2\}} \Big\{ n_{(j)} \sim N_j, n_{(l)} = \pm n_{(3)}, \sum_{j\neq l,3} \pm n_{(j)} =0 \Big\}\Big| &\les 
    \Big|\{ n_{(3)} \sim N_3, n_{(j)} \sim N_{(j)}\}_{j \ge 4}\Big| \\
    &\les N_{(3)}^2 \prod_{j=4}^{2k+2} N_{(4)}^{2} \\
    & \les N_{(1)}^{1+} N_{(3)}^{1+} \prod_{j=4}^{2k+2} N_{(4)}^{2},
\end{align*}
which is the RHS of \eqref{Ekappacounting}. Moreover, by \eqref{counting}, we have that 
\begin{align*}
    &\Big|\bigcup_{n_{(4)}, \dotsc, n_{(2k+2)}} S_{\sum_{j=4}^{2k+2} \pm n_{(j)}, \kappa - \sum_{j=4}^{2k+2} \pm |n_{(j)}|^2}(N_{(1)}, N_{(2)}, N_{(3)})\Big| \\
    & \les \prod_{j=4}^{2k+2} N_{(4)}^{2} \sup_{n_{(j)} \sim N_{(j)} : j \ge 4} | S_{\sum_{j=4}^{2k+2} \pm n_{(j)}, \kappa - \sum_{j=4}^{2k+2} \pm |n_{(j)}|^2}(N_{(1)}, N_{(2)}, N_{(3)})| \\
    &\les N_{(1)}^{1+} N_{(3)}^{1+} \prod_{j=4}^{2k+2} N_{(4)}^{2}.
\end{align*}
Therefore, we obtain \eqref{Ekappacounting}, which concludes the proof of \eqref{butcherdeng}.

In order to show \eqref{butcherdengpairing}, we proceed in the same way, replacing the set $E_\kappa$ with 
  \begin{equation*}
    \begin{multlined}
        E_\kappa^p:= \Big\{ (n_1, \dotsc, n_{k+1}, m_1, \dotsc, m_{k+1}): 
        n_j \sim N_j, m_j \sim M_j, n_{(1)} \pm n_{(2)} = 0\\ \sum_{j=1}^{k+1} n_j + m_j = 0, \sum_{j=1}^{k+1} |n_j|^2 - |m_j|^2 = \kappa\Big \}.
    \end{multlined}
    \end{equation*}
We have the trivial inclusion
$$ E_\kappa^p \subseteq \Big \{ (n_1, \dotsc, n_{k+1}, m_1, \dotsc, m_{k+1}): n_j \sim N_j, m_j \sim M_j, n_{(1)} \pm n_{(2)} = 0, \sum_{j=1}^{k+1} n_j + m_j = 0\Big\},$$
from which we immediately deduce that 
$$ |E_\kappa^p| \les N_{(1)}^2 
 \prod_{j=4}^{2k+2} N_{(4)}^{2},$$
 which replaces \eqref{Ekappacounting} and leads to \eqref{butcherdengpairing}.
\end{proof}
We we will also need the following variant of the result above, which allows for the presence of a bounded Fourier multiplier. Since the proof is identical, we omit it.
\begin{proposition}
    Let $f_1, f_2, \dotsc f_{k+1} \in X_T$, $ g_{1},\dotsc, g_{k+1} \in \cj X_T$. Let 
    $$ \Psi: (\Z^2)^{2(k+1)} \to \R$$
    be a bounded function, and define the multilinear operator 
    \begin{equation} \label{PsiDef}
        \Psi(f_1,g_1,\dotsc, f_{k+1}, g_{k+1}) = \sum_{\substack{n_1 + \dotsb n_{k+1} \\
        + m_1 + \dotsb + m_{k+1} =0}} \Psi(n_1,m_1,\dotsc, n_{k+1},m_{k+1}) \prod_{j=1}^{k+1} \ft{f_j}(n_j) \ft{g_j}(n_j).
    \end{equation}
    Let $ N_{(1)} \ge N_{(2)} \ge \dotsb N_{(2k+2)} $ be a rearrangement of $N_1,\dotsc,N_{k+1}, M_1,\dotsc,M_{k+1},$ let $n_{(1)}, \dotsc, n_{(2k+2)}$ be the corresponding rearrangement of  $n_1,\dotsc, n_{k+1}, m_{1}, \dotsc, m_{k+1}$, and let $f_{(1)}, \dotsc, f_{(2k+2)}$ be the corresponding rearrangement of $f_1, f_2, \dotsc f_{k+1}, g_{1},\dotsc, g_{k+1}$. 
    Moreover, let 
    \begin{equation} \label{multiplierPaired}
        \Psi[f_{(1)},f_{(2)}](n_j: j \neq (1), (2)) = \sum_{n_{(1)} + n_{(2)} = 0} \Psi(n_1,\dotsc,n_{2k+2}) \ft{P_{N_{(1)}}f_{(1)}}(n_{(1)}) \ft{P_{N_{(2)}}f_{(1)}}(n_{(2)}).
    \end{equation}
    Then for every $0\leq t\leq T$
    \begin{equation} \label{butcherDengPsi}
    \begin{aligned}
        &\Big|\int_0^t \Psi(P_{N_1}f_1(\tau),P_{M_1}g_1(\tau),\dotsc, P_{N_{k+1}}f_{k+1}(\tau), P_{M_{k+1}}g_{k+1}(\tau)) d \tau\\
        &\phantom{XXX} - \int_0^T \Psi[f_{(1)}(\tau), f_{(2)}(\tau)](P_{N_{(3)}} f_{(3)(\tau)},\dotsc, P_{N_{(2k+2)}} f_{(2k+2)}(\tau))\Big| \\
        &\les_T \|\Psi\|_{l^\infty} N_{(1)}^{1-2\s+} N_{(3)}^{-\s} \prod_{j=4}^{2k+2} N_{(4)}^{2-\s} \prod_{j=1}^k \| f_j\|_{X_T}\| g_j\|_{\cj X_T}
    \end{aligned}
    \end{equation}
    and
    \begin{equation} \label{butcherDengPsi2}
    \begin{aligned}
        &\Big|\int_0^t \Psi[f_{(1)}(\tau), f_{(2)}(\tau)](P_{N_{(3)}} f_{(3)(\tau)},\dotsc, P_{N_{(2k+2)}} f_{(2k+2)}(\tau)) d \tau\Big| \\
        &\les_T \|\Psi\|_{l^\infty} N_{(1)}^{2-2\s+} N_{(3)}^{-\s} \prod_{j=4}^{2k+2} N_{(4)}^{2-\s} \prod_{j=1}^k \| f_j\|_{X_T}\| g_j\|_{\cj X_T}.
    \end{aligned}
    \end{equation}
    
\end{proposition}
\subsection{Large phase regime}
\begin{proposition}\label{prop:butchercrispi}
    Let $f_1,f_2,\dotsc,f_k \in \cj{X}_T$ and $g_1,g_2,\dotsc,g_k \in X_T$. Let $N_1,\dotsc,N_k$, $M_1,\dotsc,M_k$ be dyadic numbers with $N_1 \ge N_2\ge \dotsb $, and $M_1\ge M_2 \ge \dotsb$. Finally, suppose that 
    \begin{equation*}
        N_2^2 \gg k M_1^2.
    \end{equation*}
    Then for every $0\le t \le T$,
    $$ \Big|\int_0^t \prod_{j=1}^k P_{N_j}f_j(\tau,x) P_{M_j}g_j(\tau,x) d\tau dx\Big| \les_T \frac{1}{N_1^4}\prod_{j=1}^k (N_j M_j)^{2-\s} \|f_j\|_{ X_T} \|g_j\|_{\cj X_T}.  $$
    Moreover, if $\Psi: (\Z^2)^{2k} \to \R$ is a bounded multiplier, we have that for every $0\le t \le T$,
    \begin{equation} \label{butchercrispi}
    \begin{aligned}
        &\Big|\int_0^t \Psi\big( P_{N_1}f_1(\tau),P_{M_1}g_1(\tau),\dotsc,P_{N_k}f_k(\tau), P_{M_k}g_k(\tau) d\tau dx\big)\Big| \\
        &\les_T \frac{\|\Psi\|_{\ell^\infty}}{N_1^4}\prod_{j=1}^k (N_j M_j)^{2-\s} \|f_j\|_{ X_T} \|g_j\|_{\cj X_T},
    \end{aligned}
    \end{equation} 
    where $\Psi$ is defined as in \eqref{PsiDef}.
\end{proposition}
\begin{proof}
    Since the proof for the general case is identical, we focus on the case $\Psi \equiv 1$. Define
    \begin{equation} \label{FjC1}
       F_j = e^{it\Delta} f_j, \quad G_j = e^{-it\Delta} g_j
    \end{equation}
    By definition, we have that 
      $$ \|F_j\|_{C^1([0,T], {\mathcal F}L^{\s,\infty})} \les \|f_j\|_{\cj X_T}, \quad \|G_j\|_{C^1([0,T], {\mathcal F}L^{\s,\infty})} \les \|g_j\|_{\cj X_T}.$$
    Therefore, writing the integral in Fourier variables and integrating by parts, calling 
    $$ \phi = \phi(n_1, \dotsc, n_k, m_1, \dotsc, m_k) = |n_1|^2+ \dotsb + |n_k|^2 - |m_1|^2- \dotsb - |m_k|^2, $$ 
    we get 
    \begin{align*}
        &\int_0^t \prod_{j=1}^k P_{N_j}f_j(\tau,x) P_{M_j}g_j(\tau,x) d\tau dx \\
        &= \int_0^t \sum_{\substack{n_1 +\dotsb + n_k \\
        + m_1 + \dotsb m_k =0,\\
        n_j \sim N_j, m_j \sim M_j}} 
        {e^{i\tau \phi}} \prod_{j=1}^k \ft{F_j}(\tau,n_j)\ft{G_j}(\tau,m_j)d\tau \\
        &= \left[\sum_{\substack{n_1 +\dotsb + n_k \\
        + m_1 + \dotsb m_k =0,\\
        n_j \sim N_j, m_j \sim M_j}}
        \frac{e^{i\tau\phi} -1}{i \phi} \prod_{j=1}^k \ft{F_j}(\tau,n_j)\ft{G_j}(\tau,m_j) \right]_0^t \\
        &\phantom{XX} - \int_0^t \sum_{\substack{n_1 +\dotsb + n_k \\
        + m_1 + \dotsb m_k =0,\\
        n_j \sim N_j, m_j \sim M_j}}
        \frac{e^{i\tau\phi} -1}{i \phi} \partial_\tau\left(\prod_{j=1}^k \ft{F_j}(\tau,n_j)\ft{G_j}(\tau,m_j)\right) d\tau.
    \end{align*}
    By the assumptions on $N_j,M_j$, we have that 
    $$ |\phi| \gtrsim N_1^2.$$ 
    Moreover, by Leibnitz and \eqref{FjC1}, we have that 
    $$\left|\partial_\tau\left(\prod_{j=1}^k \ft{F_j}(\tau,n_j)\ft{G_j}(\tau,m_j)\right)\right| \les \prod_{j=1}^{k}N_j^{-\s}M_j^{-\s} \|f_j\|_{X_T} \|g_j\|_{\cj X_T}. $$
    Therefore, by triangle inequality, we obtain that 
    \begin{align*}
         &\left|\int_0^t \prod_{j=1}^k P_{N_j}f_j(\tau,x) P_{M_j}g_j(\tau,x) d\tau dx\right|\\
         &\les_T \sum_{\substack{n_1 +\dotsb + n_k \\
        + m_1 + \dotsb m_k =0,\\
        n_j \sim N_j, m_j \sim M_j}} \frac{1}{N_1^2} \prod_{j=1}^{k}N_j^{-\s}M_j^{-\s} \|f_j\|_{ X_T} \|g_j\|_{\cj X_T} \\
        &\les \prod_{j=2}^{k} N_j^2 \times \prod_{j=1}^k M_j^2 \times \frac{1}{N_1^2} \times \prod_{j=1}^{k}N_j^{-\s}M_j^{-\s} \|f_j\|_{ X_T} \|g_j\|_{\cj X_T} \\
        &= \frac{1}{N_1^4}\prod_{j=1}^k (N_j M_j)^{2-\s} \|f_j\|_{ X_T} \|g_j\|_{\cj X_T}.
    \end{align*}
\end{proof}

\subsection{Probabilistic bounds}
As we will see in the following sections, the previous deterministic bounds can be used to estimate every term of the derivative of the modified 
introduced along section \ref{modifiedenergy}, with the exception of the following 
\begin{equation}\label{totime0}
    \Big(\int \partial_j u D^s \cj u\Big) \Big(\int D^{s-2} \partial_j \cj u u |u|^{2(k-1)}\Big). 
\end{equation}
For this term, we will need the following probabilistic estimate.  
\begin{lemma}[Pairing term I, random estimate]
    Let 
    $$ u(x) = \sum_{n \in \Z^2} \frac{g_n}{\jb{n}^s} e^{in\cdot x}.  $$
    We have that 
    \begin{align}\label{pairing1}
        \E\Big| \int P_{N}(\partial_j u) P_N(D^s \cj u)\Big|^2 \les N^{4-2s}.
    \end{align}
\end{lemma}
\begin{proof}
    We have that 
    $$ \int P_{N}(\partial_j u) P_N(D^s \cj u) = \sum_{n \sim N} i n_j \frac{|g_n|^2}{\jb{n}^s} = \sum_{n \sim N} i n_j \frac{|g_n|^2 -1}{\jb{n}^s}.$$
    Therefore, 
    \begin{align*}
        \E\Big| \int P_{N}(\partial_j u) P_N(D^s \cj u)\Big|^2 \les \sum_{n \sim N} \frac{1}{\jb{n}^{2s-2}} \les N^{4 -2s}.
    \end{align*}
\end{proof}
On top of the previous, we will also need the following bound.
\begin{lemma}[Pairing term II, deterministic estimate]
    Let $\max(s-1,2)< \s < s$. Then 
    \begin{equation} \label{paring2}
        \Big| \int D^{s-2} \partial_j \cj u u |u|^{2(k-1)} \Big| \les \|u\|_{{\mathcal F}L^{\s,\infty}}^{2k}. 
    \end{equation} 
\end{lemma}
\begin{proof}
By integration by parts we get
\begin{equation*}\Big| \int D^{s-2} \partial_j \cj u u |u|^{2(k-1)} \Big|\leq \|u\|_{H^{s-2}}\|u\|_{H^1} \|u\|_{L^\infty}^{2k-2}\end{equation*}
and we conclude by the inclusions:
$${\mathcal F}L^{\s,\infty}\subset H^{s-2}, \quad {\mathcal F}L^{\s,\infty}\subset H^{1}, \quad {\mathcal F}L^{\s,\infty}\subset L^{\infty}.$$

\end{proof}

We are now ready to show the required exponential bound for \eqref{totime0}.
\begin{lemma}\label{proofQin0}
For every $p, R>0$ we have the bound
\begin{equation}
\int_{B_R} \exp\Big (p\big |\big (\int D^s \bar u \nabla u\big) \cdot \big (\int |u|^{2k-2}u^2 D^{s-2} \bar u \nabla \bar u\big )\big|\Big )d\mu(u)   <+\infty
\end{equation}
where
$$ B_R = \{ u \in \FL^{\s, \infty}: \|u\|_{\FL^{\s, \infty}} \le R\}. $$
\begin{proof}
We denote 
$$ \int P_{N}(\nabla u) P_N(D^s \cj u) = B_N(u).$$
Note that, for every $N_0$ dyadic, we have the  bound 
\begin{equation*}
    |B_{N_0}(u)| \les N_0^{s+1} \| u\|_{\FL^{\s, \infty}}^2.
\end{equation*}
Therefore, defining 
$$\QQ_{>{N_0}}(u) =\big (\int D^s \bar u \nabla u - \sum_{M \le N_0} B_M(u) \big) \cdot \big (\int |u|^{2k-2}u^2 D^{s-2} \bar u \nabla \bar u\big ), $$
in view of \eqref{paring2}, it is enough to show that there exists $N_0$ dyadic such that 
\begin{equation}
    \int_{B_R} \exp( p |\QQ_{>N_0}(u)|) d \mu(u) < \infty.
\end{equation}
By \eqref{pairing1}, together with standard Wiener Chaos estimates (see \cite[Theorem I.22]{Simon}), for every $p\ge 1$, we have that 
\begin{equation}
     \E\Big|\int D^s \bar u \nabla u - \sum_{M \le N_0} B_M(u) \Big|^p \le p^p C^p N_0^{p (2-s)},
\end{equation}
for some universal constant $C > 0$.
Therefore, from \eqref{paring2}, we have that
\begin{equation} \label{QQ>toprove}
   \E |\QQ_{>N_0}(u)|^p \1_{B_R}(u) \le (C')^p R^{2kp} \E\Big|\int D^s \bar u \nabla u - \sum_{M \le N_0} B_M(u) \Big|^p  \le p^p (CC')^p N_0^{p (2-s)}R^{2kp},
\end{equation}
where $C'$ is the implicit constant in \eqref{paring2}. Therefore, we have 
\begin{align*}
    &\int_{B_R} \exp( p |\QQ_{>N_0}(u)|) d \mu(u) \\
    &= \sum_{h=0}^\infty  \int_{B_R} \frac{p^h |\QQ_{>N_0}(u)|^h}{h!} d\mu(u) \\
    &\le \frac{p^h (CC')^h N_0^{h (2-s)}R^{2hp}}{h!},
\end{align*}
so by choosing $N_0 \gg 1$ such that 
$$ p (CC') N_0^{(2-s)}R^{2k} < e, $$
the sum above converges and we get \eqref{QQ>toprove}.
\end{proof}
\end{lemma}
\section{Energy bounds}
From \eqref{energyestimate}, we can decompose the derivative of the energy in the following way 
\begin{align}
    &\frac{d}{dt}\Big (\|D^s u\|_{L^2}^2 -
k \Re (D^{s-2} \bar u u^{k+1} \bar u ^{k-1}, D^s u)\Big) \notag \\
    &= (\mathcal C^s u, D^s u) \tag{I} \label{E:1}\\
    &\phantom{=\ }-2k (k-1)
\Im (D^{s-2} \bar u u^{k+1}  \Delta \bar u \bar u^{k-2}, D^s u)  \tag{II} \label{E:2} \\
&\phantom{=\ }- 2k(k+1) 
\Im (D^{s-2} \nabla \bar u \nabla u u^{k} \bar u^{k-1}, D^{s} u) \tag{III} \label{E:3}\\
&\phantom{=\ }-  2k(k-1) 
\Im (D^{s-2} \nabla \bar u u^{k+1} \nabla \bar u \bar u^{k-2}, D^{s} u) \tag{IV} \label{E:4}\\
&\phantom{=\ }- 2 k (k+1)(k-1)\Im (D^{s-2} \bar u \nabla  u\nabla \bar u  u^{k} \bar u^{k-2}, D^{s} u) \tag{V} \label{E:5}\\
&\phantom{=\ }\begin{rcases}
&\phantom{=\ }+2 \Im (\nabla C^{s-2}(\bar u) |u|^{2k-2}u^2, \nabla D^{s-2} u)\notag\\
&\phantom{=\ }+(2k+2)\Im (D^{s-2}\bar u \nabla (|u|^{2k})|u|^{2k-2}u^2, \nabla D^{s-2} u)
\notag\\&\phantom{=\ }+k
\Im (D^{s-2}u \nabla (|u|^{2k-2} \bar u^{2}) |u|^{2k-2}u^2, \nabla D^{s-2} u)
\notag\\&\phantom{=\ }+
\Im (D^{s-2}(\bar u|u|^{2k} ) \nabla (|u|^{2k-2}u^2), \nabla D^{s-2} u)
\notag\\&\phantom{=\ }-2\Im (D^{s-2} \bar u \nabla(|u|^{4k-2} u^2), \nabla D^{s-2}u)
\notag\\&\phantom{=\ } +\Im (D^{s-2}\bar u \nabla (|u|^{2k-2} u^2), \nabla D^{s-2} (u|u|^{2k}))\notag\\
&\phantom{=\ } +k\Im (\nabla D^{s-2}\bar u |u|^{2k-2} u^2, D^{s-2} \bar u \nabla(u^{k+1}\bar u^{k-1}))
\end{rcases}\tag{VI}\label{E:6} \\
&\phantom{=\ } + 2k \Im(\nabla D^{s-2}\cj u |u|^{4k-2} u^2, \nabla D^{s-2} u) \tag{VII}\label{E:7}.
\end{align}
We note that each term \eqref{E:1}--\eqref{E:7} can be written as 
\begin{equation} \label{Psieta}
    \Im\big(\Psi_{\eta}(u,\cj{u},\dotsc,u,\cj{u})\big) = \Im\left(\sum_{N_1,N_2,\dotsc,N_{k_\eta}} \Psi_{\eta}(P_{N_1}u,\cj{P_{N_2}u},\dotsc,P_{N_{k_\eta - 1}}u,\cj{P_{N_{k_\eta - 1}}u})\right)
\end{equation}
where 
$\eta \in \{ \rm{I, II, III, IV, V, VI, VII}\},$
and $k_\eta = 2k+2$ if $\eta \in \{ \rm{I, II, III, IV, V}\},$ and $k_\eta = 4k+2$ otherwise. Moreover, since we always have 
$$ |\Psi_\eta| \les N_{(1)}^{2s}, $$
by \eqref{butcherDengPsi} and \eqref{butcherDengPsi2}, we deduce that 
\begin{align}
    \left|\int_0^T \sum_{N_{(3)} \gtrsim N_{(1)}} \Psi_{\eta}(P_{N_1}u,\cj{P_{N_2}u},\dotsc,P_{N_{k_\eta - 1}}u,\cj{P_{N_{k_\eta}}u}) d\tau \right| & \les_T \|u\|_{X_T}^{k_\eta}\sum_{N_{(3)} \gtrsim N_{(1)}} N_{(1)}^{2s} N_{(1)}^{2-2\s+} N_{(3)}^{-\s} \notag \\
    &\les \|u\|_{X_T}^{k_\eta} \sum  N_{(1)}^{2-s+} \notag \\
    &\les  \|u\|_{X_T}^{k_\eta}, \label{Psieta1}
\end{align}
since $s > 2$. Therefore, now for each term we bound the RHS of \eqref{Psieta} restricted to the condition $N_{(3)} \ll N_{(1)}$.
\subsection{Term \texorpdfstring{\eqref{E:1}}{(I)}} In this case we have that 
\begin{equation} \label{psiI}
    \Psi_{\rm{I}}(n_1,\dotsc,n_{2k+2}) = |n_{2k+2}|^s\left(\Big|\sum_{j=1}^{2k+1} n_j\Big|^s - \sum_{j=1}^{2k+1} |n_j|^s\right).
\end{equation}
We observe right away that if $|n_{2k+2}| \les N_{(3)},$ then 
$$ |\Psi_{\rm{I}}| \les N_{(1)}^s N_{(3)}^s, $$
and so by \eqref{butcherdeng} and \eqref{butcherdengpairing}, we obtain that 
\begin{equation} \label{contodanonripetereI}
    \begin{aligned}
    &\sum_{N_{(3)}\ll N_{(1)}, N_{2k+2} \les N_{(3)}} |\Psi_{\rm{I}}(P_{N_1}u,\cj{P_{N_2}u},\dotsc,P_{N_{2k+1}}u,\cj{P_{N_{2k+2}}u})| \\&\les \|u\|_{X_T}^{2k+2}  \sum_{N_{(3)}\ll N_{(1)}} N_{(1)}^s N_{(3)}^s N_{(1)}^{2-2\s+} N_{(3)}^{-\s} \\
    & \les \|u\|_{X_T}^{2k+2} \sum_{N_{(3)}\ll N_{(1)}} N_{(1)}^{2-s+} N_{(3)}^{0+} \\
    &\les \|u\|_{X_T}^{2k+2}.
\end{aligned}
\end{equation}
Therefore, in the following, we work in the case $|n_{2k+2}| \sim N_{(1)}$.
We write 
\begin{align*}
    \Psi_{\rm{I}}(n_1,\dotsc,n_{2k+2}) &= |n_{2k+2}|^s\left(\Big|\sum_{j=1}^{2k+1} n_j\Big|^s - |n_{(1)}|^s\right) + \Psi_{\rm{I}}^2(n_1,\dotsc,n_{2k+2}),
\end{align*}
where $|n_{(1)}|\ge \dotsb \ge |n_{(2k+1)}|$, and $\{n_{(j)}\}_{1\le j\le 2k+1}$ is a rearrangement of $\{n_j\}_{1 \le j \le 2k+1}$.
We write
\begin{align*}
    |n_{2k+2}|^s\left(\Big|\sum_{j=1}^{2k+1} n_j\Big|^s - |n_{(1)}|^s\right) = s|n_{2k+2}|^s|n_{(1)}|^{s-2}\Big(n_{(1)} 
    \cdot \sum_{j \neq (1)} n_j\Big) + \Psi_{\rm{I}}^3(n_1,\dotsc,n_{2k+2}).
\end{align*}
\begin{lemma}
    Let $x_1,\dotsc, x_n \in \R^2$ be such that 
    $$ 1 = |x_1| \ge |x_2| \ge \dotsb. $$
    Then there exists a universal constant $C = C_n >0$ such that 
    \begin{equation} \label{realLemma}
        \left| \Big|\sum_{j=1}^n x_j \Big|^s - 1 - s x_1 \cdot \sum_{j=2}^n x_j \right| \le C |x_2|^2.
    \end{equation}
\end{lemma}
\begin{proof}
    Consider the function $q:\R^2 \to \R$ given by 
    $$q(y) = \big|x_1 + y\big|^s - 1 - s x_1\cdot y. $$
    Since $s > 2$, this function is $C^2$, and it is immediate to check that $q(0) = 0$, $\nabla q (0) = 0$. Therefore, by the mean value theorem,  for every $|y| \le n-1$, there exists a constant $C_n$ such that 
    $$|q(y)| \le C_n |y|^2. $$
    The lemma follows by picking $y = \sum_{j=2}^n x_j$.
\end{proof}
\begin{corollary}
    Let $n_1, \dotsc, n_{2k+1} \in \Z^2$, and let $n_{(1)}, \dotsc, n_{(2k+1)}$ be any rearrangement of the $n_j$ such that 
    $$|n_{(1)}| \ge \dotsc \ge |n_{(2k+1)}|. $$
    Then 
    \begin{equation}\label{commutatorDecomposition}
        \left|\Psi_{\rm I}^2 + \Psi_{\rm I}^3\right| = |n_{2k+2}|^s\left|\Big|\sum_{j=1}^{2k+1} n_j\Big|^s - \sum_{j=1}^{2k+1} |n_j|^s - s |n_{(1)}|^{s-2}n_{(1)}\cdot \sum_{j=2}^{2k+1} n_{(j)}\right| \les |n_{(1)}|^{2s-2} |n_{(2)}|^2. 
    \end{equation}
\end{corollary}
\begin{proof}
    Applying \eqref{realLemma} to 
    $$ x_j := \frac{n_{(j)}}{|n_{(1)}|}, $$
    we obtain that 
$$  \left|\Big|\sum_{j=1}^{2k+1} n_j\Big|^s - |n_{(1)}|^s - s |n_{(1)}|^{s-2}n_{(1)}\cdot \sum_{j=2}^{2k+1} n_{(j)}\right| \les |n_{(1)}|^{s-2} |n_{(2)}|^2.$$
Moroever, we clearly have that 
$$\sum_{j=2}^{2k+1} |n_{(j)}|^s \les |n_{(2)}|^s \les |n_{(1)}|^{s-2} |n_{(2)}|^2. $$
Therefore, \eqref{commutatorDecomposition} follows by summing these two inequalities. \end{proof}
In order to complete the bound, we need to distinguish two cases. \\
\textbf{Case 1:} $(1)$ is even. In this case, we use the bounds \eqref{commutatorDecomposition}, \eqref{butchercrispi} and we obtain that 
\begin{align*}
    &\left|\int_0^T \sum_{N_{(3)} \ll N_{(1)}} \Psi_{\rm I}(P_{N_1}u,\cj{P_{N_2}u},\dotsc,P_{N_{2k+1 }}u,\cj{P_{N_{2k+2}}u}) d\tau \right| \\
    &\les \|u\|_{X_T}^{2k+2} \sum_{N_{(3)} \ll N_{(1)}} N_{(1)}^{2s-1} N_{(3)}  \frac1{N_{(1)}^4}\prod_{j=1}^{2k+2} N_j^{2-\s} \\
    &\les \|u\|_{X_T}^{2k+2} \sum_{N_{(3)} \ll N_{(1)}} N_{(1)}^{-1+} N_{(3)}^{3-s+} \\
    &\les \|u\|_{X_T}^{2k+2},
\end{align*}
since $s > 2$.\\
\textbf{Case 2:} $(1)$ is odd. In this case, we need to exploit \eqref{butcherDengPsi}, so we need to understand the term in \eqref{multiplierPaired}. Recalling that we are working in the situation where $|n_{2k+2}| \sim N_{(1)}, $ we have that 
\begin{align*}
            & \Psi_{\rm I}[u_{(1)},\cj{u}_{2k+2}](n_j: j \neq (1), 2k+2) \\
            &= \sum_{n_{2k+2} + n_{(1)} = 0} \Psi_{I}(n_1,\dotsc,n_{2k+2}) \ft{P_{N_{(1)}}u_{(1)}}(n_{(1)}) \ft{P_{N_{2k+2}}\cj{u}_{2k+2}}(n_{2k+2}).
\end{align*}
Under the condition $\sum_{j=1}^{2k+2} n_j = 0$, since $n_{2k+2} + n_{(1)} = 0$, we also have that 
$\sum_{j=1, j \neq (1)}^{2k+2} n_j =0$. Therefore, by \eqref{psiI}, we have that 
\begin{align*}
    |\Psi_{\rm{I}}| \ind_{n_{2k+2} + n_{(1)} = 0} &= \left||n_{2k+2}|^s \Big(|n_{(1)}|^s - \sum_{j=1}^{2k+1} |n_j|^s\Big)\right| \ind_{n_{2k+2} + n_{(1)} = 0} \\
    &= \left||n_{2k+2}|^s \Big(\sum_{j=1, j\neq (1)}^{2k+1} |n_j|^s\Big)\right| \ind_{n_{2k+2} + n_{(1)} = 0} \\
    &\les N_{(1)}^s N_{(3)}^s.
\end{align*}
In particular, by \eqref{butcherDengPsi2}, proceeding as in \eqref{contodanonripetereI}, we obtain that
\begin{align*}
        &\Big|\int_0^T \sum_{N_{(3)} \ll N_{(1)}, (1) \text{ odd}}\Psi_{\rm I}[u_{(1)},\cj{u}_{2k+2}](P_{N_j} u, P_{N_k} \cj u )_{j,k \neq (1)}(\tau)) d \tau\Big| \\
        &\les_T \|u\|_{X_T}^{2k+2}.
\end{align*}
Therefore, in view of \eqref{commutatorDecomposition}, by \eqref{butcherDengPsi}, we have that 
\begin{align*}
    & \left|\int_0^T \sum_{N_{(3)} \ll N_{(1)}, N_{(1)} \sim N_{2k+2}, (1) \text{ odd} }(\Psi_{\rm I}^2 + \Psi_{\rm I}^3) (P_{N_1}u,\cj{P_{N_2}u},\dotsc,P_{N_{2k+1 }}u,\cj{P_{N_{2k+2}}u}) \right| \\
    &\les_T \|u\|_{X_T}^{2k+2}\Big(1 + \sum_{N_{(3)} \ll N_{(1)} } N_{(1)}^{2s-2} N_{(3)}^{2} N_{(1)}^{1 - 2\s} N_{(3)}^{-\s} \Big) \\
    &\les \|u\|_{X_T}^{2k+2}\Big(1 + \sum_{N_{(3)} \ll N_{(1)} } N_{(1)}^{-1+} N_{(3)}^{2-s+} \Big),
\end{align*}
since $s>2$. Therefore, we are only left with estimating the multiplier 
$$ s|n_{2k+2}|^s|n_{(1)}|^{s-2}\Big(n_{(1)} 
    \cdot \sum_{j \neq (1)} n_j\Big).$$
We note that we can write this term in physical space as 
$$ s \int D^s P_{N_{2k+2}} \cj{u} D^{s-2} \nabla P_{N_{(1)}}u \cdot \nabla \big(\prod_{h' \neq (1) \text{ odd}} P_{N_{h'}} u \prod_{h'' \neq (1) \text{ even}}P_{N_{h''}} \cj{u}\big).  $$
Recalling that we only need to bound the imaginary part of the expression above, we have that 
\begin{align*}
& \sum_{N_j: N_j \ll N_{(1)}, j \neq (1), 2k+2}\Im\left(s \int D^s P_{N_{2k+2}} \cj{u} D^{s-2} \nabla P_{N_{(1)}}u \cdot \nabla \big(\prod_{h' \neq (1) \text{ odd}} P_{N_{h'}} u \prod_{h'' \neq (1) \text{ even}}P_{N_{h''}} \cj{u}\big)\right) \\
& = \Im\left(s \int D^s P_{N_{2k+2}} \cj{u} D^{s-2} \nabla P_{N_{(1)}}u \cdot \nabla (|P_{\ll N_{(1)}} u|^{2k})\right) \\
& =\frac{s}{2i}\Big( \int \big(D^s P_{N_{2k+2}} \cj{u} D^{s-2} \nabla P_{N_{(1)}}u  
    -D^s P_{N_{2k+2}} {u} D^{s-2} \nabla P_{N_{(1)}}\cj u\big)  \cdot \nabla (|P_{\ll N_{(1)}} u|^{2k})\Big).
\end{align*}
Moreover, the multiplier associated to 
$$ \big(D^s P_{N_{2k+2}} \cj{u} D^{s-2} \nabla P_{N_{(1)}}u     -D^s P_{N_{2k+2}} {u} D^{s-2} \nabla P_{N_{(1)}}\cj u\big) $$
is 
$$ |n_{2k+2}|^s|n_{(1)}|^{s-2}n_{(1)} - |n_{(1)}|^s|n_{2k+2}|^{s-2}n_{2k+2}. $$
Since this is an homogeneous expression of degree $2s-1$, by the mean value theorem we have that 
$$ \big||n_{2k+2}|^s|n_{(1)}|^{s-2}n_{(1)} - |n_{(1)}|^s|n_{2k+2}|^{s-2}n_{2k+2}\big| \les N_{(1)}^{2s-2}N_{(3)}. $$ 
Therefore, by \eqref{butcherDengPsi}, we have that 
\begin{align*}
&\begin{multlined}
        \sum_{N_{(3)} \ll N_{(1)}}\bigg|\int_0^T \frac{s}{2i}\Big( \int \big(D^s P_{N_{2k+2}} \cj{u} D^{s-2} \nabla P_{N_{(1)}}u  
    -D^s P_{N_{2k+2}} {u} D^{s-2} \nabla P_{N_{(1)}}\cj u\big)  \cdot \nabla (|P_{\ll N_{(1)}} u|^{2k})\Big)\bigg|
\end{multlined} \\
&\les \|u\|_{X_T}^{2k+2} \sum_{N_{(3)} \ll N_{(1)}} N_{(1)}^{2s-2}N_{(3)}^2 N_{(1)}^{1 - 2\s} N_{(3)}^{-\s} \\
&\les \|u\|_{X_T}^{2k+2} \sum_{N_{(3)} \ll N_{(1)}} N_{(1)}^{-1+}N_{(3)}^{2-s+} \\
&\les \|u\|_{X_T}^{2k+2},
\end{align*}
which concludes the estimate for this term. In particular, we have that 
\begin{equation} \label{E:1bound}
    \left|\int_0^T \Psi_{\rm{I}}(u(\tau),\cj{u}(\tau),\dotsc,u(\tau),\cj{u}(\tau))\right|  \les \|u\|_{X_T}^{2k+2}. 
\end{equation}
\subsection{Term \texorpdfstring{\eqref{E:2}}{(II)}} 
In this case, we have that 
\begin{equation} \label{psiII}
    \Psi_{\rm{II}} \sim |n_2|^{s}|n_4|^2|n_6|^{s-2}.
\end{equation}
We further decompose the RHS of \eqref{Psieta} the condition $N_{(3)} \ll N_{(1)}$ in the following way
\begin{align*}
    &\Big(
    \sum_{\substack{N_{(3)} \ll N_{(1)}\\ N_2, \max(N_4,N_6) \sim N_{(1)}}} + \sum_{\substack{N_{(3)} \ll N_{(1)}\\ \min(N_2, \max(N_4,N_6)) \ll N_{(1)}}}\Big)  \Psi_{\rm{II}}(P_{N_1}u,\cj{P_{N_2}u},\dotsc,P_{N_{2k+1}}u,\cj{P_{N_{2k+2}}u}) \\
    &=: \big(
    \Psi_{\rm{II}}^{2} + \Psi_{\rm{II}}^{3}\big) (u,\cj{u},\dotsc,u,\cj{u}).
\end{align*}
For $\Psi_{\rm{II}}^2$, we rely on \eqref{butchercrispi}, and obtain 
\begin{align*}
    \left|\int_0^T \Psi_{\rm{II}}^{2}(u,\cj{u},\dotsc,u,\cj{u})\right| & \les \|u\|_{X_T}^{2k+2} \sum_{\substack{N_{(3)} \ll N_{(1)}\\ N_2, \max(N_4,N_6) \sim N_{(1)}}}
    N_{(1)}^{\max(s+2,2s-2)}N_{(3)}^{\min(s-2,2)} \frac{1}{N_{(1)}^4} \prod_{j=1}^{2k+2} N_{(j)}^{2-\s}\\
    &\les \|u\|_{X_T}^{2k+2} \sum N_{(1)}^{\max(2-s,-2)+} N_{(3)}^{\min(0,4-s)+} \prod_{j=4}^{2k+2} N_{(j)}^{2-\s} \\
    &\les \|u\|_{X_T}^{2k+2} \sum N_{(1)}^{2-s+} \\
    &\les \|u\|_{X_T}^{2k+2},
\end{align*}
since $s>2$. Finally, for $\Psi_{\rm{II}}^3$, we use once again \eqref{butcherdengpairing}, and obtain that 
\begin{align*}
     \left|\int_0^T \Psi_{\rm{II}}^{3}(u,\cj{u},\dotsc,u,\cj{u})\right| & \les \|u\|_{X_T}^{2k+2} 
     \sum_{\substack{N_{(3)} \ll N_{(1)}\\ \min(N_2, \max(N_4,N_6)) \ll N_{(1)}}} N_{(1)}^{s} N_{(3)}^s   
     N_{(1)}^{2-2\s+} N_{(3)}^{-\s} \prod_{j=4}^{2k+2} N_{(j)}^{2-\s}\\
     &\les \|u\|_{X_T}^{2k+2} \sum N_{(1)}^{2-s+} N_{(3)}^{0+} \prod_{j=4}^{2k+2} N_{(j)}^{2-s+} \\
     &\les \|u\|_{X_T}^{2k+2},
\end{align*}
since once again $s>2$. In particular, we deduce that 
\begin{equation} \label{E:2bound}
    \left|\int_0^T \Psi_{\rm{II}}(u(\tau),\cj{u}(\tau),\dotsc,u(\tau),\cj{u}(\tau))\right|  \les \|u\|_{X_T}^{2k+2}. 
\end{equation}
\subsection{Term \texorpdfstring{\eqref{E:3}}{(III)}}
In this case, we have that 
$$\Psi_{\rm{III}} \sim n_1\cdot n_2|n_2|^{s-2} |n_4|^{s}$$
We further decompose the RHS of \eqref{Psieta} the condition $N_{(3)} \ll N_{(1)}$ in the following way,
\begin{align*} 
    &\Big(
   \sum_{\substack{N_{(3)} \ll N_{(1)}\\ N_2,N_4 \sim N_{(1)}}} + 
    \sum_{\substack{N_{(3)} \ll N_{(1)}\\ N_4 \ll N_{(1)} \text{ or }\\
    N_4 \sim N_{(1)}, \max(N_1,N_2) \ll N_{(1)}}} +
    \sum_{\substack{N_{(3)} \ll N_{(1)}\\ N_1,N_4 \sim N_{(1)}}}
   \Big)
    \Psi_{\rm{III}}(P_{N_1}u,\cj{P_{N_2}u},\dotsc,P_{N_{2k+1}}u,\cj{P_{N_{2k+2}}u}) \\
    &=: \big(
    \Psi_{\rm{III}}^{2} + \Psi_{\rm{III}}^{3}
    + \Psi_{\rm{III}}^4\big) (u,\cj{u},\dotsc,u,\cj{u}).
\end{align*}
For the first term, we use \eqref{butchercrispi}, and deduce that 
\begin{align*}
    \left|\int_0^T \Psi_{\rm{III}}^{2}(u,\cj{u},\dotsc,u,\cj{u})\right| & \les \|u\|_{X_T}^{2k+2} \sum_{\substack{N_{(3)} \ll N_{(1)}\\ N_2,N_4 \sim N_{(1)}}}
    N_{(1)}^{2s-1}N_{(3)} \frac{1}{N_{(1)}^4} \prod_{j=1}^{2k+2} N_{(j)}^{2-\s}\\
    &\les \|u\|_{X_T}^{2k+2} \sum_{\substack{N_{(3)} \ll N_{(1)}\\ N_2,N_4 \sim N_{(1)}}}
    N_{(1)}^{-1+}N_{(3)}^{3-s+} \prod_{j=4}^{2k+2} N_{(j)}^{2-\s} \\
    &\les \|u\|_{X_T}^{2k+2} \sum N_{(1)}^{2-s+} \\
    &\les \|u\|_{X_T}^{2k+2},
\end{align*}
since $s>2$. For $\Psi_{\rm{III}}^3$, we use \eqref{butcherdeng} and \eqref{butcherdengpairing}, and obtain 
\begin{align*}
    \left|\int_0^T \Psi_{\rm{III}}^{3}(u,\cj{u},\dotsc,u,\cj{u})\right| & \les \|u\|_{X_T}^{2k+2} \sum_{\substack{N_{(3)} \ll N_{(1)}\\ N_4 \ll N_{(1)} \text{ or }\\
    N_4 \sim N_{(1)}, \max(N_1,N_2) \ll N_{(1)}}}
    N_{(1)}^{s}N_{(3)}^s N_{(1)}^{2-2\s +} N_{(3)}^{-\s} \prod_{j=4}^{2k+2} N_{(j)}^{2-\s}\\
    &\les \|u\|_{X_T}^{2k+2} \sum_{\substack{N_{(3)} \ll N_{(1)}}}
    N_{(1)}^{2-s+}N_{(3)}^{0+} \prod_{j=4}^{2k+2} N_{(j)}^{2-\s} \\
    &\les \|u\|_{X_T}^{2k+2} \sum N_{(1)}^{2-s+} \\
    &\les \|u\|_{X_T}^{2k+2},
\end{align*}
since $s>2$ once again. For $\Psi_{\rm{III}}^4$, we use \eqref{butcherdeng} and deduce that 
\begin{align*}
    &\left|\int_0^T \left(\Psi_{\rm{III}}^{4}(u,\cj{u},\dotsc,u,\cj{u}) + 2k(k+1) \Big(\int D^s \cj{u} \nabla u\Big)\cdot \int |u|^{2k-1} u D^{s-2}\nabla \cj{u}   \right) d \tau \right| \\
    & \les \|u\|_{X_T}^{2k+2}     \sum_{\substack{N_{(3)} \ll N_{(1)}\\ N_1,N_4 \sim N_{(1)}}}
    N_{(1)}^{s+1}N_{(3)}^{s-1} N_{(1)}^{1-2\s +} N_{(3)}^{1-\s} \prod_{j=4}^{2k+2} N_{(j)}^{2-\s}\\
    &\les \|u\|_{X_T}^{2k+2} \sum_{\substack{N_{(3)} \ll N_{(1)}\\ N_1,N_4 \sim N_{(1)}}}
    N_{(1)}^{2-s+}N_{(3)}^{0+} \prod_{j=4}^{2k+2} N_{(j)}^{2-\s} \\
    &\les \|u\|_{X_T}^{2k+2} \sum N_{(1)}^{2-s+} \\
    &\les \|u\|_{X_T}^{2k+2},
\end{align*}
Together with \eqref{Psieta1}, we obtain 
\begin{equation} \label{E:3bound}
\begin{aligned}
        &\left|\int_0^T \left(\Psi_{\rm{III}}(u,\cj{u},\dotsc,u,\cj{u})+ 2k(k+1)(k-1) \Big(\int D^s \cj{u} \nabla u\Big)\cdot \int |u|^{2k-2} u^2 D^{s-2} \cj{u} \nabla \cj{u}  \right) d \tau \right|  \\
        &\les \|u\|_{X_T}^{2k+2}.
\end{aligned}
\end{equation}
\subsection{Term \texorpdfstring{\eqref{E:4}}{(IV)}}
In this case, we have that 
\begin{equation} \label{psiIV}
    \Psi_{\rm{IV}} \sim n_2\cdot n_4|n_4|^{s-2} |n_6|^{s}.
\end{equation}
We further decompose the RHS of \eqref{Psieta} the condition $N_{(3)} \ll N_{(1)}$ in the following way,
\begin{align*} 
    &\Big(
   \sum_{\substack{N_{(3)} \ll N_{(1)}\\ \max(N_2,N_4),N_6 \sim N_{(1)}}} + \sum_{\substack{N_{(3)} \ll N_{(1)}\\ \min(\max(N_2,N_4),N_6) \ll N_{(1)}}}\Big)
    \Psi_{\rm{IV}}(P_{N_1}u,\cj{P_{N_2}u},\dotsc,P_{N_{2k+1}}u,\cj{P_{N_{2k+2}}u}) \\
    &=: \big(
    \Psi_{\rm{IV}}^{2} + \Psi_{\rm{IV}}^{3}\big) (u,\cj{u},\dotsc,u,\cj{u}).
\end{align*}
For the first term, we use \eqref{butchercrispi}, and deduce that 
\begin{align*}
    \left|\int_0^T \Psi_{\rm{IV}}^{2}(u,\cj{u},\dotsc,u,\cj{u})\right| & \les \|u\|_{X_T}^{2k+2} \sum_{\substack{N_{(3)} \ll N_{(1)}\\ \max(N_2,N_4),N_6 \sim N_{(1)}}}
    N_{(1)}^{2s-1}N_{(3)} \frac{1}{N_{(1)}^4} \prod_{j=1}^{2k+2} N_{(j)}^{2-\s}\\
    &\les \|u\|_{X_T}^{2k+2} \sum_{\substack{N_{(3)} \ll N_{(1)}\\ \max(N_2,N_4),N_6 \sim N_{(1)}}}
    N_{(1)}^{-1+}N_{(3)}^{3-s+} \prod_{j=4}^{2k+2} N_{(j)}^{2-\s} \\
    &\les \|u\|_{X_T}^{2k+2} \sum N_{(1)}^{2-s+} \\
    &\les \|u\|_{X_T}^{2k+2},
\end{align*}
since $s>2$. For $\Psi_{\rm{IV}}^3$, we use \eqref{butcherdeng} and \eqref{butcherdengpairing}, and obtain 
\begin{align*}
    \left|\int_0^T \Psi_{\rm{IV}}^{3}(u,\cj{u},\dotsc,u,\cj{u})\right| & \les \|u\|_{X_T}^{2k+2} \sum_{\substack{N_{(3)} \ll N_{(1)}\\ \min(\max(N_2,N_4),N_6) \ll N_{(1)}}}
    N_{(1)}^{s}N_{(3)}^s N_{(1)}^{2-2\s +} N_{(3)}^{-\s} \prod_{j=4}^{2k+2} N_{(j)}^{2-\s}\\
    &\les \|u\|_{X_T}^{2k+2} \sum_{\substack{N_{(3)} \ll N_{(1)}\\ \min(\max(N_2,N_4),N_6) \ll N_{(1)}}}
    N_{(1)}^{2-s+}N_{(3)}^{0+} \prod_{j=4}^{2k+2} N_{(j)}^{2-\s} \\
    &\les \|u\|_{X_T}^{2k+2} \sum N_{(1)}^{2-s+} \\
    &\les \|u\|_{X_T}^{2k+2},
\end{align*}
since $s>2$ once again. Together with \eqref{Psieta1}, we obtain 
\begin{equation} \label{E:4bound}
    \left|\int_0^T \Psi_{\rm{IV}}(u(\tau),\cj{u}(\tau),\dotsc,u(\tau),\cj{u}(\tau))\right|  \les \|u\|_{X_T}^{2k+2}. 
\end{equation}
\subsection{Term \texorpdfstring{\eqref{E:5}}{V}}
In this case, we have that 
\begin{equation}\label{psiV}
    \Psi_{\rm{V}} \sim n_1\cdot n_4 |n_2|^{s-2} |n_6|^{s}.
\end{equation}
We further decompose the RHS of \eqref{Psieta} the condition $N_{(3)} \ll N_{(1)}$ in the following way,
\begin{align*} 
    &\Big(
   \sum_{\substack{N_{(3)} \ll N_{(1)}\\ \max(N_2,N_4),N_6 \sim N_{(1)}}} + 
    \sum_{\substack{N_{(3)} \ll N_{(1)}\\ N_6 \ll N_{(1)} \text{ or }\\
    N_6 \sim N_{(1)}, \max(N_1,N_2,N_4) \ll N_{(1)}}} +
    \sum_{\substack{N_{(3)} \ll N_{(1)}\\ N_1,N_6 \sim N_{(1)}}}
   \Big)
    \Psi_{\rm{V}}(P_{N_1}u,\cj{P_{N_2}u},\dotsc,P_{N_{2k+1}}u,\cj{P_{N_{2k+2}}u}) \\
    &=: \big(
    \Psi_{\rm{V}}^{2} + \Psi_{\rm{V}}^{3}
    + \Psi_{\rm{V}}^4\big) (u,\cj{u},\dotsc,u,\cj{u}).
\end{align*}
For the first term, we use \eqref{butchercrispi}, and deduce that 
\begin{align*}
    \left|\int_0^T \Psi_{\rm{V}}^{2}(u,\cj{u},\dotsc,u,\cj{u})\right| & \les \|u\|_{X_T}^{2k+2} \sum_{N_{(3)} \ll N_{(1)}}
    N_{(1)}^{s + \max(s-2,1)}N_{(3)}^{1+\min(s-2,1)} \frac{1}{N_{(1)}^4} \prod_{j=1}^{2k+2} N_{(j)}^{2-\s}\\
    &\les \|u\|_{X_T}^{2k+2} \sum_{N_{(3)} \ll N_{(1)}}
    N_{(1)}^{\max(-2,1-s)+}N_{(3)}^{\min(1,3-s)+} \prod_{j=4}^{2k+2} N_{(j)}^{2-\s} \\
    &\les \|u\|_{X_T}^{2k+2} \sum N_{(1)}^{2-s+} \\
    &\les \|u\|_{X_T}^{2k+2},
\end{align*}
since $s>2$. For $\Psi_{\rm{V}}^3$, we use \eqref{butcherdeng} and \eqref{butcherdengpairing}, and obtain 
\begin{align*}
    \left|\int_0^T \Psi_{\rm{V}}^{3}(u,\cj{u},\dotsc,u,\cj{u})\right| & \les \|u\|_{X_T}^{2k+2} \sum_{\substack{N_{(3)} \ll N_{(1)}\\ N_6 \ll N_{(1)} \text{ or }\\
    N_6 \sim N_{(1)}, \max(N_1,N_2,N_4) \ll N_{(1)}}}
    N_{(1)}^{s}N_{(3)}^s N_{(1)}^{2-2\s +} N_{(3)}^{-\s} \prod_{j=4}^{2k+2} N_{(j)}^{2-\s}\\
    &\les \|u\|_{X_T}^{2k+2} \sum_{\substack{N_{(3)} \ll N_{(1)}\\ N_6 \ll N_{(1)} \text{ or }\\
    N_6 \sim N_{(1)}, \max(N_1,N_2,N_4) \ll N_{(1)}}} \prod_{j=4}^{2k+2} N_{(j)}^{2-\s} \\
    &\les \|u\|_{X_T}^{2k+2} \sum N_{(1)}^{2-s+} \\
    &\les \|u\|_{X_T}^{2k+2},
\end{align*}
since $s>2$ once again. For $\Psi_{\rm{V}}^3$, we use \eqref{butcherdeng} and deduce that 
\begin{align*}
    &\left|\int_0^T \left(\Psi_{\rm{V}}^{4}(u,\cj{u},\dotsc,u,\cj{u}) + 2k(k+1)(k-1) \Big(\int D^s \cj{u} \nabla u\Big)\cdot \int |u|^{2k-2} u^2 D^{s-2} \cj{u} \nabla \cj{u}  \right) d \tau \right| \\
    & \les \|u\|_{X_T}^{2k+2}     \sum_{\substack{N_{(3)} \ll N_{(1)}\\ N_1,N_6 \sim N_{(1)}}}
    N_{(1)}^{s+1}N_{(3)}^{s-1} N_{(1)}^{1-2\s +} N_{(3)}^{1-\s} \prod_{j=4}^{2k+2} N_{(j)}^{2-\s}\\
    &\les \|u\|_{X_T}^{2k+2} \sum_{\substack{N_{(3)} \ll N_{(1)}\\ N_6 \ll N_{(1)}}}
    N_{(1)}^{2-s+}N_{(3)}^{0+} \prod_{j=4}^{2k+2} N_{(j)}^{2-\s} \\
    &\les \|u\|_{X_T}^{2k+2} \sum N_{(1)}^{2-s+} \\
    &\les \|u\|_{X_T}^{2k+2},
\end{align*}
Together with \eqref{Psieta1}, we obtain 
\begin{equation} \label{E:5bound}
\begin{aligned}
        &\left|\int_0^T \left(\Psi_{\rm{V}}(u,\cj{u},\dotsc,u,\cj{u})+ 2k(k+1)(k-1) \Big(\int D^s \cj{u} \nabla u\Big)\cdot \int |u|^{2k-2} u^2 D^{s-2} \cj{u} \nabla \cj{u}  \right) d \tau \right|  \\
        &\les \|u\|_{X_T}^{2k+2}.
\end{aligned}
\end{equation}
\subsection{Term \texorpdfstring{\eqref{E:6}}{(VI)}}
In this case, we have that 
$$ |\Psi_{\rm{VI}}| \les N_{(1)}^{2s-3} N_{(3)}.$$
Therefore, by \eqref{butcherdeng} and \eqref{butcherdengpairing}, we obtain 
\begin{equation} \label{E:6bound}
\begin{aligned}
        &\left|\int_0^T \Psi_{\rm{VI}}(u,\cj{u},\dotsc,u,\cj{u}) d \tau \right| \\
        &\les \|u\|_{X_T}^{2k+2} \sum N_{(1)}^{2s-3} N_{(3)} N_{(1)}^{2-2\s} N_{(3)}^{-\s} \\
        &\les \|u\|_{X_T}^{2k+2} \sum N_{(1)}^{-1+} N_{(3)}^{1-s+} \\
        &\les \|u\|_{X_T}^{2k+2}.
\end{aligned}
\end{equation}
\subsection{Term \texorpdfstring{\eqref{E:7}}{(VII)}}
In this case, we have that 
\begin{equation} \label{psiVII}
    |\Psi_{\rm{VII}}| \sim |n_2|^{s-1} |n_4|^{s-1}. 
\end{equation}
We further decompose the RHS of \eqref{psiVII} under the condition $N_{(3)} \ll N_{(1)}$ in the following way,
\begin{align*} 
    &\Big(
   \sum_{\substack{N_{(3)} \ll N_{(1)}\\ \min(N_2,N_4)\ll N_{(1)}}} + \sum_{\substack{N_{(3)} \ll N_{(1)}\\ N_2, N_4 \sim N_{(1)}}}\Big)
    \Psi_{\rm{VII}}(P_{N_1}u,\cj{P_{N_2}u},\dotsc,P_{N_{2k+1}}u,\cj{P_{N_{2k+2}}u}) \\
    &=: \big(
    \Psi_{\rm{VII}}^{2} + \Psi_{\rm{VII}}^{3}\big) (u,\cj{u},\dotsc,u,\cj{u}).
\end{align*}
For the first term, we use \eqref{butcherdeng} and \eqref{butcherdengpairing}, and obtain that 
\begin{align*}
    &\left| \int_0^T \Psi_{\rm{VII}}^{2}(u,\cj{u},\dotsc,u,\cj{u}) d \tau \right| \\
    &\les \|u\|_{X_T}^{2k+2} \sum_{\substack{N_{(3)} \ll N_{(1)}\\ \min(N_2,N_4)\ll N_{(1)}}} N_{2}^{s-1}N_{4}^{s-1} N_{(1)}^{2-2\s+} N_{(3)}^{-\s} \\
    &\les \|u\|_{X_T}^{2k+2} \sum N_{(1)}^{s-1}N_{(3)}^{s-1} N_{(1)}^{2-2s+} N_{(3)}^{-s+} \\
    &\les \|u\|_{X_T}^{2k+2} \sum N_{(1)}^{1-s+}N_{(3)}^{-1+} \\
    &\les \|u\|_{X_T}^{2k+2},
\end{align*}
since $s > 2$. For the second term, we use \eqref{butchercrispi}, and deduce 
\begin{align*}
    &\left| \int_0^T \Psi_{\rm{VII}}^{3}(u,\cj{u},\dotsc,u,\cj{u}) d \tau \right| \\
    &\les \|u\|_{X_T}^{2k+2} \sum_{\substack{N_{(3)} \ll N_{(1)}\\ N_2,N_4\sim N_{(1)}}} N_{2}^{s-1}N_{4}^{s-1} \frac{1}{N_{(1)}^4}\prod_{j=1}^{4k+2} N_j^{2-\s} \\
    &\les \|u\|_{X_T}^{2k+2} \sum N_{(1)}^{2s-2} N_{(1)}^{-2\s}\\
    &\les \|u\|_{X_T}^{2k+2} \sum N_{(1)}^{-2+} \\
    &\les \|u\|_{X_T}^{2k+2},
\end{align*}
where we used once again that $s > 2$. Together with \eqref{Psieta1}, we obtain 
\begin{equation} \label{E:7bound}
\begin{aligned}
        \left|\int_0^T \Psi_{\rm{VII}}(u,\cj{u},\dotsc,u,\cj{u}) d \tau\right| 
        &\les_T \|u\|_{X_T}^{2k+2}.
\end{aligned}
\end{equation}
\section{Proof of the main theorem}
Let $S$ be as in \eqref{Sdef}, and let the Gaussian measure $\mu=\mu_s$ as in \eqref{defmus}. We first show that the measure 
\begin{equation} \label{rhodeffinal}
    \rho = \exp(-S(u)) d \mu
\end{equation}
is quasi-invariant. To this purpose, we aim to apply Proposition \ref{abstractquasiinvariance}. We then check all the assumptions 1-6. \\
\textbf{Assumption \ref{ass:1}:} We take $X = \FL^{\s,\infty}$ for $2< \s <s$ with $s-\s$ small enough. We have that $H = H^s$. The assumption then holds by definition of $\mu$ and $S$. \\
\textbf{Assumption \ref{ass:2}:} We pick 
$$ E_n = \{ f \in H^s: \ft f (m) = 0 \text{ for every } |m| > 2^n\}, $$
and so $P_n = P_{\le 2^n}$. We then define $\Phi_n$ to be the flow generated by the systems of ODEs on $E_n$ given by
\begin{equation} \label{Phindef}
    i \partial_t u + \Delta u = P_{n}(u |u|^{2k}). 
\end{equation}
The various properties of the flow follow from the well-posedness results in Section \ref{cauchyfl}, together with a standard approximation argument. \\
\textbf{Assumption \ref{ass3}} The fact that $\Phi_n$ preserves the Lebesgue measure of $E_n$ follows by Liouville's theorem applied to \eqref{Phindef}. We define 
\begin{gather}
\begin{multlined} \label{Q1def}
        \QQ_1^n(u)=-2k(k-1)(k+1)\big (\int D^s P_{\le N} \bar  u \nabla P_{\le N} u\big) \\
        \phantom{XXXXXXXXXXXXX}\times \big (\int |P_{\le N} u|^{2k-2}(P_{\le N} u)^2 D^{s-2} P_{\le N} \bar u \nabla P_{\le N} \bar u\big )\big|\Big ),
\end{multlined} \\ \label{Q2def}
    \QQ_2^n(u)=\frac 12 ({\rm I_n}+{\rm II_n}
+{\rm III}_n+{\rm IV}_n+{\rm V}_n+{\rm VI}_n+{\rm VII}_n)-\QQ_1^n(u),
\end{gather}
where
${\rm I_n}, {\rm II_n}
, {\rm III}_n, {\rm IV}_n, {\rm V}_n,{\rm VI}_n, {\rm VII}_n$
are respectively the expressions 
\eqref{E:1}, \eqref{E:2}, \eqref{E:3}, \eqref{E:4}, \eqref{E:5}, \eqref{E:6}, \eqref{E:7} computed at $P_{\le N} u$ (for $N = 2^n$).
Then \eqref{norm_derivative} follows from (the finite-dimensional analogue of) Proposition \ref{modifiedenergy}. \\
\textbf{Assumption \ref{ass_global}:} This is an immediate consequence of the global well-posedness result of Section \ref{cauchyfl}.\\
\textbf{Assumption \ref{ass5}:}
We have that \eqref{Rconvergence} is the content of Lemma \ref{lem:Sconvergence}, while \eqref{expSLploc} follows by \eqref{lem:Sexploc}.\\
\textbf{Assumption \ref{ass6}:}
We have that \eqref{Qin0} follows from \eqref{proofQin0} and Lemma \eqref{lem:Sexploc} together with H\"older's inequality. The bound
\eqref{spacetimeQ} follows from the fact that $\int_0^T \QQ_2^n \circ \Phi_t d t$ is bounded on every ball $B_R$, which in turn is a consequence of \eqref{E:1bound},
\eqref{E:2bound}, \eqref{E:3bound}, \eqref{E:4bound}, \eqref{E:5bound}, \eqref{E:6bound}, and \eqref{E:7bound}.
\begin{proof}[Proof of Theorem \ref{thm:main}]
    Since the measure $\rho$ in \eqref{rhodeffinal} is equivalent to $\mu_s$, by Proposition \ref{abstractquasiinvariance} we have that 
    $$ \Phi_{t\#}\mu_s \ll \Phi_{t\#}\rho_s \ll \rho_s \ll \mu_s.$$
    We now move to showing \eqref{mudensitybound}. Let 
    $$ g_t := \frac{d \Phi_{t\#} \rho}{d \rho}. $$
    For a functional $F$, we have that 
    \begin{align*}
        \int F \circ \Phi_t d \mu & = \int F \circ \Phi_t \exp(S) d \rho \\ 
        &= \int \big(F \exp(S\circ \Phi_{-t}\big) \circ \Phi_t d \rho \\
        &= \int \big(F \exp(S\circ \Phi_{-t}\big) g_t d\rho \\
        &= \int F \exp(S\circ \Phi_{-t} - S)g_t d\mu,
    \end{align*}
    from which we deduce that 
    \begin{equation}
        f_t = g_t \exp(S\circ \Phi_{-t}) \exp(- S).
    \end{equation}
    In particular, by H\"older, \eqref{localDensityBound}, and \eqref{Smomentbound}, we have that 
    \begin{align*}
        &\sup_{|t| \le T} \| f_t \1_{\|u\|_{X} \le R}\|_{L^p(\mu_s)}^p \\
        &= \sup_{|t| \le T} \big\|g_t \exp(S\circ \Phi_{-t}) \exp(- S) \1_{\|u\|_{X} \le R}\big\|_{L^p(\mu_s)}^p \\
        &= \sup_{|t| \le T} \int_{B_R} g_t^p  \exp(pS\circ \Phi_{-t}) \exp(- (p-1) S) d \rho \\
        &\le \sup_{|t| \le T} \Big(\int_{B_R} g_t^{3p-1} d\rho \int_{B_R} g_t \exp(3p S\circ \Phi_{-t}) d \rho \int_{B_R} \exp(-(3(p-1)) S) d \rho \Big)^\frac13 \\
        & \le \sup_{|t| \le T} \Big(\int_{B_R} g_t^{3p-1} d\rho \int_{B_{R'(T)}} \exp(3p S) d \rho \int_{B_R} \exp(-(3p-2) S) d \mu \Big)^\frac13 \\
        &\le \Big( \big(C_{6Tp}(R'(T))\big)^{\frac{18p^2}{(6p-1)(6p-2)}} \big(D_{2(3p-1)}(R'(T),T))^\frac{3p}{6p-1} \big(\int_{B_{R'(T)}} \exp(3p|S|) d \mu)^2 \Big)^\frac13 \\
        & < \infty.
    \end{align*}
\end{proof}


\begin{thebibliography}{99}

\bibitem{BPTV} D. Berti, F. Planchon, N. Ttzvetkov, N. Visciglia,
{\it New bounds on the high {S}obolev norms of the 1{D} {NLS}
              solutions},
J. Differential Equations, 430,
(2025) 113202, 34 pp.


  \bibitem{Bo96}
J. Bourgain, {\it Invariant measures for the 
2D-defocusing nonlinear Schr\"odinger equation},
Comm. Math. Phys. 176(2): 421-445 (1996).

\bibitem{BGT} N. Burq, P. G\'erard, N. Tzvetkov, {\it Strichartz inequalities and the nonlinear Schr\"odinger equation on compact manifolds},
Amer. J. Math. 126 (2004), no. 3, 569--605.



\bibitem{BTh}
N.~Burq, L.~Thomann,
{\it Almost sure scattering for the one dimensional nonlinear Schr\"{o}dinger equation},
Mem. Amer. Math. Soc. 296(2024), no.1480, vii+87 pp.






\bibitem{CT}
J. Coe, L. Tolomeo,
{\it Sharp quasi-invariance threshold for the cubic Szeg\"{o} equation}, 
	to appear in Anal. PDE.
	
\bibitem{CHT}
J. Coe, M. Hairer, L. Tolomeo,
 {\it Quasi-Gaussianity of the 2D stochastic Navier-Stokes equations,} arXiv:2510.13460 [math.PR].







\bibitem{DT2020} 
A.~Debussche, Y.~Tsutsumi,
{\it  Quasi-invariance of Gaussian measures transported by the
	cubic NLS with third-order dispersion on $\mathbf T$}, 
J. Funct. Anal. 281 (2021), no. 3, 109032, 23 pp.


\bibitem{DNY3}
 Y.~Deng, A.~R.~Nahmod, H.~Yue,
 {\it Random tensors, propagation of randomness, and nonlinear dispersive equations,}
Invent. Math., (228):539--686, 2022.


	
\bibitem{FS}
J.~Forlano, K.~Seong,
{\it Transport of Gaussian measures under the flow of one-dimensional fractional nonlinear Schr\"{o}dinger equations,} Comm. PDE, 47 (2022), no. 6, 1296--1337.


\bibitem{FT2}
J. Forlano, L. Tolomeo,
{\it Quasi-invariance of Gaussian measures of negative regularity for fractional nonlinear Schr\"{o}dinger equations}, J. Eur. Math. Soc. (2025), published online first.

\bibitem{FT3}
J. Forlano, L. Tolomeo,{\it 
 Quasi-invariance of the Gaussian measure for the two-dimensional stochastic cubic nonlinear wave equation}, arXiv:2409.20451.

\bibitem{FT}
 J.~Forlano, W.~J.~Trenberth, 
 {\it On the transport of Gaussian measures under the one-dimensional fractional nonlinear Schr\"odinger equations,} 
	Ann. Inst. H. Poincar\'e Anal. Non Lin\'eaire 36 (2019), no. 7, 1987--2025.


\bibitem{GLT1}
G.~Genovese, R.~Luc\'{a}, N.~Tzvetkov,
	{\it Transport of Gaussian measures with exponential cut-off for Hamiltonian PDEs},
J. Anal. Math.150(2023), no.2, 737--787.

\bibitem{GLT2}
G.~Genovese, R.~Luc\'{a}, N.~Tzvetkov,
{\it Quasi-invariance of Gaussian measures for the periodic Benjamin-Ono-BBM equation}, 
Stoch. Partial Differ. Equ. Anal. Comput.11(2023), no.2, 651--684.






\bibitem{GOTW}
T.~Gunaratnam, T.~Oh, N.~Tzvetkov, H.~Weber,
{\it  Quasi-invariant Gaussian measures for the nonlinear wave equation in three dimensions},
Probab. Math. Phys. 3 (2022), no. 2, 343--379.



\bibitem{HOV} M.~Hayashi, T.~Ozawa, and N.~Visciglia, \emph{Global $H^2$-solutions for the generalized derivative NLS on $\T$}, SIAM J. Math. Anal. \textbf{57} (2025), 1483--1501.





\bibitem{Knezevitch1}
A. Knezevitch,
{\it Transport of low regularity Gaussian measures for the 1d quintic nonlinear Schr\"{o}dinger equation},
	 Nonlinear Differ. Equ. Appl. 32, 45 (2025).

\bibitem{Knezevitch2}
A. Knezevitch,
{\it Qualitative quasi-invariance of low regularity Gaussian measures for the 1d quintic nonlinear Schr\"odinger equation},
	 arXiv:2502.17094.

\bibitem{Knezevitch3}
A. Knezevitch,
{\it Quantitative quasi-invariance of Gaussian measures below the energy level for the 1D generalized nonlinear Schr\"odinger equation and application to global well-posedness},
Proceedings of the Royal Society of Edinburgh: Section A Mathematics. Published online 2025:1-33.











\bibitem{OS}
T.~Oh,  K.~Seong, 
{\it  Quasi-invariant Gaussian measures for the cubic fourth order 
	nonlinear Schr\"odinger equation 
	in negative Sobolev spaces,}
J. Funct. Anal. 281 (2021), no. 9, 109150, 49 pp.

\bibitem{OST}
T.~Oh, P.~Sosoe, N.~Tzvetkov,
{\it  An optimal regularity result on the quasi-invariant Gaussian measures for the cubic fourth order nonlinear 
	Schr\"odinger equation,} 
J. \'Ec. polytech. Math. 5
(2018), 793--841. 


\bibitem{OTT}
T.~Oh, Y.~Tsutsumi, N.~Tzvetkov,
{\it  Quasi-invariant Gaussian measures for the cubic nonlinear Schr\"odinger equation with third order dispersion}, C. R. Math. Acad. Sci. Paris 357 (2019), no. 4, 366--381. 

  \bibitem{OTz}
T.~Oh, N.~Tzvetkov,
{\it Quasi-invariant Gaussian measures for the cubic fourth order nonlinear Schr\"odinger equation}, 
Probab. Theory Related Fields 169 (2017), 1121--1168. 


\bibitem{OTz2}
T.~Oh, N.~Tzvetkov,
{\it Quasi-invariant Gaussian measures for the two-dimensional defocusing cubic nonlinear wave
	equation}, 
J. Eur. Math. Soc. 22 (2020), no. 6, 1785--1826.





\bibitem{PTV} 
F.~Planchon, N.~Tzvetkov, N.~Visciglia, 
{\it  Transport of Gaussian measures by the 
	flow of the nonlinear Schr\"odinger equation}, 
Math. Ann. 378 (2020), no. 1-2, 389--423.



\bibitem{PTV2} 
F.~Planchon, N.~Tzvetkov, N.~Visciglia, 
{\it Modified energies for the periodic generalized KdV equation and applications},
Ann. Inst. H. Poincar\'e C Anal. Non Lin\'eaire 40 (2023), no.4, 863--917.

\bibitem{PTV_first} F. Planchon, N. Tzvetkov, and N. Visciglia, {\it  On the growth of Sobolev norms for NLS on 2- and 3-dimensional manifolds},  Anal. PDE \textbf{10} (2017), 1123--1147. 



\bibitem{PTVHarm} F. Planchon, N. Tzvetkov, and N. Visciglia, \textit{Growth of Sobolev norms for \(2 d\) NLS with harmonic potential}, Rev. Mat. Iberoam. \textbf{39} (2023), 1405--1436. 


\bibitem{STX}
P.~Sosoe, W.~Trenberth, T.~Xian,
{\it Quasi-invariance of fractional Gaussian fields by the nonlinear wave equation with polynomial nonlinearity}, 
Differential Integral Equations 33 (2020), no. 7-8, 393--430.


\bibitem{Simon}
B. Simon, {\it The $P(\Phi)2$ Euclidean (quantum) field theory}, Princeton Series in Physics. Princeton University
Press, Princeton, N.J., 1974. xx+392 pp.


	
\bibitem{ST}
C.~Sun, N.~Tzvetkov,
{\it Quasi-invariance of Gaussian measures for the 3d energy critical nonlinear Schr\"{o}dinger equation},
Comm. Pure Appl. Math. 78 (2025), no. 12, 2305--2353.	

    \bibitem{ST2} C.~Sun, N.~Tzvetkov, Almost sure global nonlinear smoothing for the 2D NLS. Preprint
	\bibitem{ThTz}
L. Thomann, N. Tzvetkov, {\it Gibbs measure for the periodic derivative nonlinear Schr\"{o}dinger equation},
Nonlinearity 23 (2010), no. 11, 2771--2791.
	


 





\bibitem{Tz}
N.~Tzvetkov, {\it Quasi-invariant Gaussian measures for one dimensional Hamiltonian PDE's,}
Forum Math. Sigma 3 (2015), e28, 35 pp. 



\end{thebibliography}
\end{document}